\documentclass[11pt]{amsart}
\usepackage[utf8]{inputenc} 
\usepackage[T1]{fontenc}
\usepackage{url}
\usepackage{tcolorbox}
\usepackage{cite}
\usepackage{amsmath} 

\usepackage[normalem]{ulem}
\usepackage[margin=1.3in]{geometry}
                             
\usepackage{soul,times,amsfonts,amssymb,amsthm,bbm,mathrsfs,xcolor,graphicx,enumitem,booktabs,euscript,mathtools,nicefrac,float,graphicx,enumitem}

\usepackage{cite}
\usepackage[font=footnotesize]{subcaption}
\usepackage{caption}

\usepackage[unicode]{hyperref}
\hypersetup{
    colorlinks,
    linkcolor={red!80!black},
    citecolor={red!80!black},
    urlcolor={red!80!black}
}
\usepackage[capitalize]{cleveref}
\newcommand\redout{\bgroup\markoverwith{\textcolor{red}{\rule[0.5ex]{2pt}{0.8pt}}}\ULon}

\newtheorem{theorem}{Theorem}[section]
\newtheorem{lemma}[theorem]{Lemma}

\newtheorem{proposition}[theorem]{Proposition}

\theoremstyle{remark}
\newtheorem{remark}{Remark}
\newtheorem{definition}{Definition}[section]

\makeatletter
\AtEndEnvironment{remark}{\hfill$\triangleleft$}
\makeatother
\makeatletter
\AtEndEnvironment{example}{\hfill$\triangleleft$}
\makeatother

\newcommand\nc\newcommand
\nc\bfa{{\boldsymbol a}}\nc\bfA{{\boldsymbol A}}\nc\sA{{\EuScript A}}\nc\cA{{\mathcal A}}
\nc\bfb{{\boldsymbol b}}\nc\bfB{{\boldsymbol B}}\nc\cB{{\mathcal B}}\nc\sB{{\EuScript B}}
\nc\bfc{{\boldsymbol c}}\nc\bfC{{\boldsymbol C}}\nc\cC{{\mathscr C}}
\nc\bfd{{\boldsymbol d}}\nc\bfD{{\boldsymbol D}}\nc\cD{{\mathscr D}}
\nc\bfe{{\boldsymbol e}}\nc\bfE{{\boldsymbol E}}\nc\cE{{\mathcal E}}\nc\sE{{\mathscr E}}
\nc\bff{{\boldsymbol f}}\nc\bfF{{\boldsymbol F}}\nc\cF{{\mathscr F}}
\nc\bfg{{\boldsymbol g}}\nc\bfG{{\boldsymbol G}}\nc\cG{{\EuScript G}}\nc\sG{{\mathscr G}}
\nc\bfh{{\boldsymbol h}}\nc\bfH{{\boldsymbol H}}\nc\cH{{\mathcal H}}\nc\sH{{\mathscr H}}
\nc\bfi{{\boldsymbol i}}\nc\bfI{{\boldsymbol I}}\nc\cI{{\EuScript I}}\nc\sI{{\mathscr I}}
\nc\bfj{{\boldsymbol j}}\nc\bfJ{{\boldsymbol J}}\nc\cJ{{\EuScript J}}
\nc\bfk{{\boldsymbol k}}\nc\bfK{{\boldsymbol K}}\nc\cK{{\EuScript K}}
\nc\bfl{{\boldsymbol l}}\nc\bfL{{\boldsymbol L}}\nc\cL{{\EuScript L}}
\nc\bfm{{\boldsymbol m}}\nc\bfM{{\boldsymbol M}}\nc\cM{{\EuScript M}}\nc\sM{{\mathscr M}}
\nc\bfn{{\boldsymbol n}}\nc\bfN{{\boldsymbol N}}\nc\cN{{\EuScript N}}
\nc\bfo{{\boldsymbol o}}\nc\bfO{{\boldsymbol O}}\nc\cO{{\EuScript O}}\nc\sO{{\mathscr O}}
\nc\bfp{{\boldsymbol p}}\nc\bfP{{\boldsymbol P}}\nc\cP{{\EuScript P}}\nc\sP{{\mathscr P}}
\nc\bfq{{\boldsymbol q}}\nc\bfQ{{\boldsymbol Q}}\nc\cQ{{\mathcal Q}}
\nc\bfr{{\boldsymbol r}}\nc\bfR{{\boldsymbol R}}\nc\cR{{\EuScript R}}\nc\sR{{\mathscr R}}
\nc\bfs{{\boldsymbol s}}\nc\bfS{{\boldsymbol S}}\nc\cS{{\EuScript S}}
\nc\bft{{\boldsymbol t}}\nc\bfT{{\boldsymbol T}}\nc\cT{{\EuScript T}}\nc\sT{{\mathscr T}}
\nc\bfu{{\boldsymbol u}}\nc\bfU{{\boldsymbol U}}\nc\cU{{\EuScript U}}
\nc\bfv{{\boldsymbol v}}\nc\bfV{{\boldsymbol V}}\nc\cV{{\mathscr V}}
\nc\bfw{{\boldsymbol w}}\nc\bfW{{\boldsymbol W}}\nc\cW{{\mathscr W}}
\nc\bfx{{\boldsymbol x}}\nc\bfX{{\boldsymbol X}}\nc\cX{{\EuScript X}}
\nc\bfy{{\boldsymbol y}}\nc\bfY{{\boldsymbol Y}}\nc\cY{{\mathscr Y}}
\nc\bfz{{\boldsymbol z}}\nc\bfZ{{\boldsymbol Z}}\nc\cZ{{\EuScript Z}}
\nc{\1}{{\mathbbm{1}}}

\nc{\remove}[1]{}

\allowdisplaybreaks

\DeclareSymbolFont{bbold}{U}{bbold}{m}{n}
\DeclareSymbolFontAlphabet{\mathbbold}{bbold}

\DeclareMathOperator{\supp}{supp}

\newcommand{\R}{{\mathbb R}}

\newcommand{\Z}{{\mathbb Z}}

\nc\torus{{\mathbb{T}}}
\nc\reals{{\mathbb R}}
\nc{\ff}{{\mathbb F}}
\nc{\PP}{{\mathbb P}}
\nc{\complex}{{\mathbb C}}

\nc\inftyeq{\overset{\infty}{=}}

\newcommand{\A}{{\mathbb A}}

\usepackage{tikz}
\usetikzlibrary{positioning}
\usetikzlibrary{calc}
\usetikzlibrary{decorations.pathreplacing}

\definecolor{newred}{HTML}{ff382e}
\definecolor{newgreen}{HTML}{549641}
\definecolor{newblue}{HTML}{4c4cfc}
\definecolor{neworange}{HTML}{c98702}

\nc\Renyi{R{\'e}nyi }
\usepackage[capitalize]{cleveref}
\title[Maximal hard-core model]{The maximal hard-core model {on the triangular lattice}}

\author[A. Barg, N. Kashyap, and G. Wang]{Alexander Barg}
\address{(AB) Univ.\ of Maryland, College Park, USA}
\email{abarg@umd.edu}
\author[]{Navin Kashyap}
\address{(NK) Indian Institute of Science, Bangalore, India}
\email{nkashyap@iisc.ac.in}
\author[]{Geyang Wang}
\address{(GW) Univ.\ of Maryland, College Park, USA}
\email{wanggy@umd.edu}

\begin{document}

\begin{abstract} 
The well-known hard-core model on the triangular lattice $\A$ is an interaction model defined on the independent sets of the lattice, parameterized by an activity parameter $\lambda > 0$. In this work, we consider an extension of this model to maximal independent sets, which we call the maximal hard-core model. We show that in the high-activity regime ($\lambda > 10^6$), the model admits at least two Gibbs measures on maximal independent sets in $\A$. In the low-activity regime ($\lambda$ close to $0$), we apply Pirogov-Sinai theory to characterize extremal periodic Gibbs measures of the model. Finally, we derive bounds on the capacity of a recoverable system on the lattice associated with the maximal hard-core model.
\end{abstract}

\maketitle

\section{Introduction: The maximal hard-core model}
Under the hard-core model in a graph, probabilities are assigned to independent sets based on the activity parameter that biases the distribution toward more dense or more sparse configurations. This model has been extensively studied in statistical mechanics and combinatorics~\cite{gaunt1965hard,dobrushin1968problem,brightwell2002,Jenssen_2024,blanca2019phase,mazel2019high,mazel2025high,he2024high}. 
Recently, we considered a version of this model that limits the sample space to {\em maximal independent sets} (MISs) in the graph \cite{wang2025maximal}. A new feature of the maximal hard-core model arises because, unlike the standard case, the maximality constraint does not allow arbitrarily sparse configurations, so nontrivial behavior arises both in the high- and low-activity regimes. In \cite{wang2025maximal}, we considered the case of $\Z^2$, and in this paper, we extend this study to the triangular lattice, a.k.a.\ the hard-hexagon model \cite{baxter1980}.
While the maximal hard-core model on $\Z^2$ has previously appeared in the physics literature  \cite{dallastaStatisticalMechanicsMaximal2009}, this work and our previous paper \cite{wang2025maximal} appear to be the first to consider its mathematical aspects. 

Consider a triangular lattice $\A\subset \R^2$ formed of triangles with side length 1. Sites (vertices) $u,v\in \A$ are called adjacent if the distance between them is 1, and 
connecting all pairs of adjacent vertices by edges defines a graph from which independent sets are sampled. An independent set can be defined as a function $\omega\in
{\{0,1\}}^\A$ such that if $\omega_u=\omega_v=1$, then the sites $u$ and $v$ are not adjacent. We think of $\omega$ as a configuration in $\A$ and write $\omega_u$ instead
of $\omega(u)$. An independent set $\omega$ is called maximal if it is not contained in another independent set, meaning that every $u\in \A$ such that $\omega_u=0$ is adjacent to at least one site $v$ with $\omega_v=1$. The hard-core model on $\A$ was recently studied in \cite{mazel2025high}, and extending it to the maximal case forms a natural mathematical problem. The maximal hard-core model is also motivated by some applications in storage systems, where it can be used as a model of storage devices, and we discuss some problems arising from this connection in the last section of the paper. Examples of MISs in $\Z^2$ and $\A$ are shown in \cref{fig:MIS}.

\begin{figure}[ht]
  \centering
  \begin{subfigure}[t]{0.2\textwidth}
  \captionsetup{width=1.4\textwidth}
    \includegraphics[width=\linewidth]{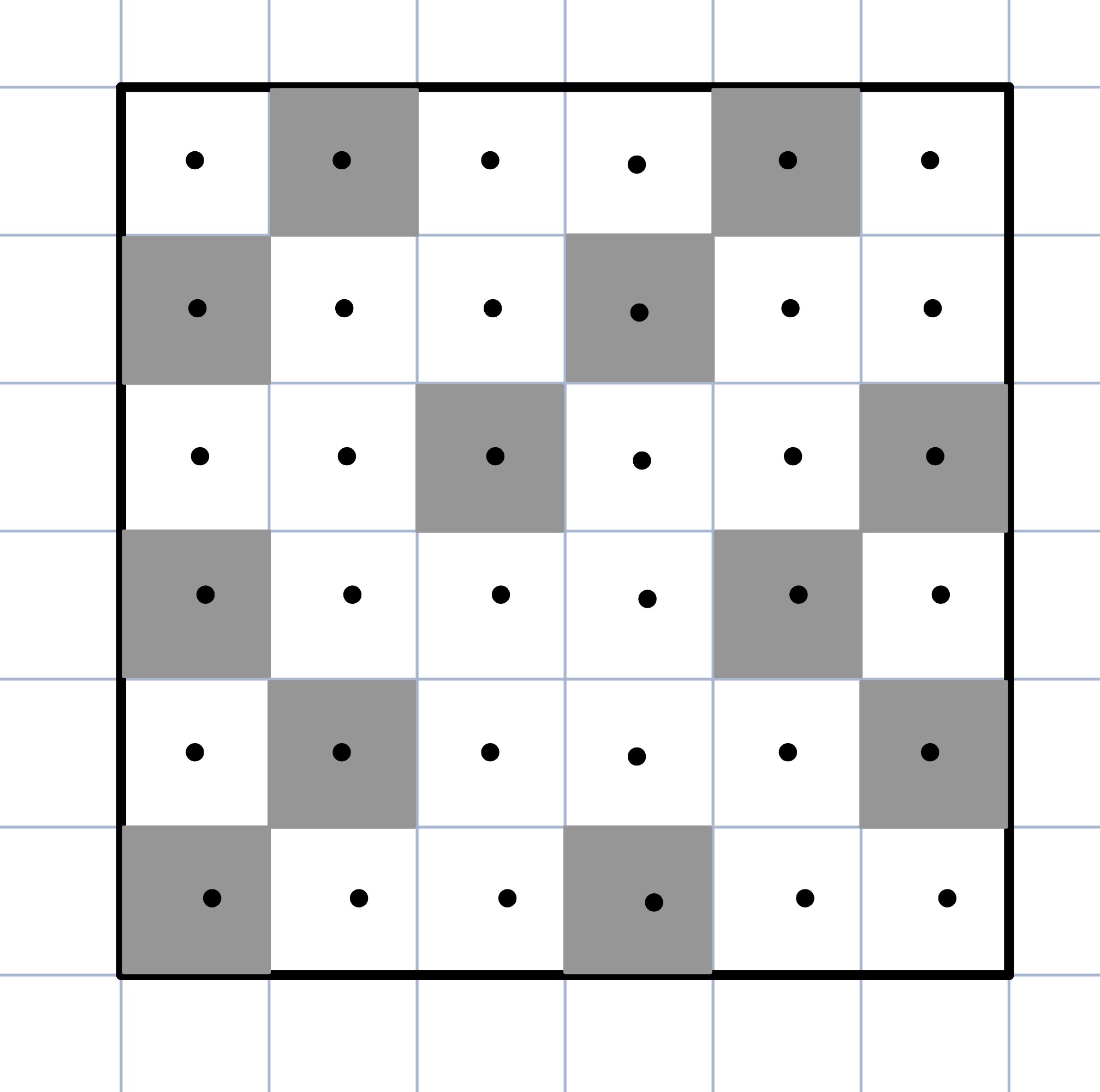}
    \subcaption{Each square is centered on a site in $\Z^2$. Occupied sites are shown by filled squares.}\label{fig:MIS-Z}
  \end{subfigure}\hspace*{1in}
 \begin{subfigure}[t]{0.27\textwidth}
 \captionsetup{width=1.4\textwidth}
 \raisebox{10pt}{\includegraphics[width=\linewidth]{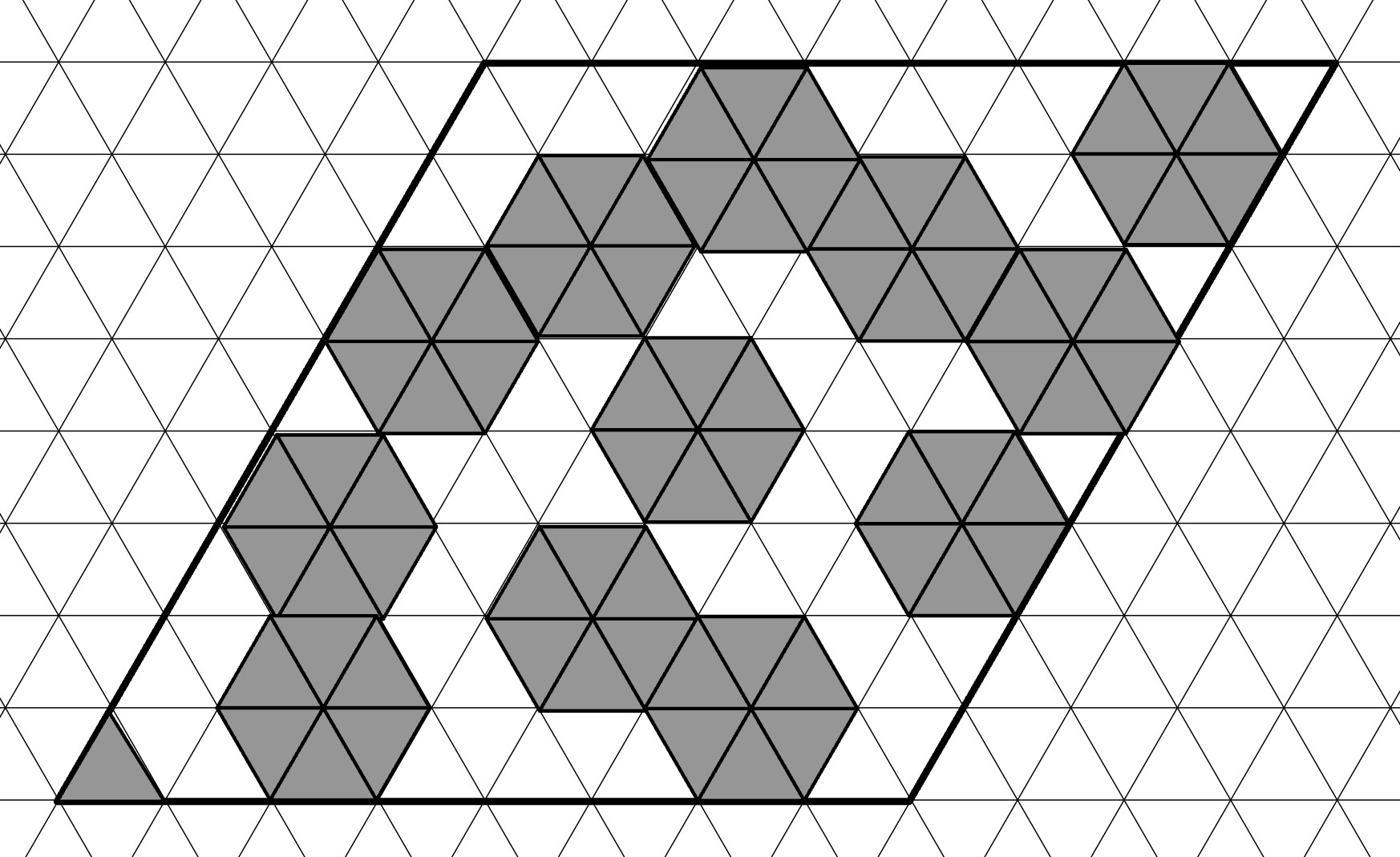}}
    \subcaption{Occupied sites are shown as centers of filled hexagons. A set of centers forms a MIS if and only if there is no room to fit another filled hexagon.}\label{fig:MIS-H}
  \end{subfigure}
\vspace*{-.1in}  \caption{MISs in finite regions in $\Z^2$ and $\A$.}\label{fig:MIS}
\end{figure}

Denote by $\cX(\A)$ the collection of all configurations that correspond to MISs. For $\omega\in\cX(\A)$, we denote by $\omega_\Lambda$ its restriction to a finite region $\Lambda\Subset \A$. Denote the complement of $\Lambda$ in $\A$ by $\Lambda^c$ and let $\omega_{\Lambda^c}:=(\omega_i)_{i\in \Lambda^c}$, $|\omega_{\Lambda}|:=|\supp(\omega_{\Lambda})|$.

\begin{definition}[The maximal hard-core model]\label{def: MHCM}
    A probability measure $\mu$ on  $\cX(\A)$ is called a \emph{maximal hard-core Gibbs measure} with activity $\lambda > 0$ if, for every finite set $\Lambda \Subset \A$, $\eta \in \cX(\A)$ and $\mu$-almost every $\omega \in \cX(\A)$, the conditional probability on $\Lambda$ satisfies     
    \begin{equation}\label{eq:def_hard-core_mis}
        \mu_\Lambda(\eta \mid \omega) = \begin{cases}
            \frac{\lambda^{|\eta_{\Lambda}|}}{\bfZ_{\Lambda, \lambda}^\omega} & \text{if } \eta_{\Lambda^c} = \omega_{\Lambda^c}, \\
            0 & \text{otherwise},
        \end{cases}
    \end{equation}
    where the partition function is given by
    \begin{equation*}
        \bfZ_{\Lambda, \lambda}^\omega = \sum_{\eta \in \cX(\A): \eta_{\Lambda^c} = \omega_{\Lambda^c}} \lambda^{|\eta_\Lambda|}.
    \end{equation*}
We note that maximal hard-core measures always exist; a formal argument for $\Z^2$ is presented in \cite[App.~C]{wang2025maximal}.
\end{definition}

As our first result, we prove that for large $\lambda$ there exist at least two Gibbs measures on $\cX(\A)$ (phase coexistence in the high-activity regime). The proof relies on the general method of  ``contour erasing'' \cite{dobrushin1968problem,blanca2019phase,wang2025maximal}.  To define contours, we analyze the subgraph of empty faces associated with a MIS and estimate
the number of contours in a finite region. Arguments similar to the ones used in Sec.~\ref{sec: phase coexistence} were earlier used in \cite{duminil2017exponential} for the analysis of the loop model on the hexagonal lattice. Switching to the low-activity regime, we identify all periodic ground states of the model and use the Pirogov-Sinai theory to claim that
each of them gives rise to a periodic Gibbs measure of the model. Finally, we recall the concept of recoverable systems and its connection to the maximal model  \cite{elishcoRecoverableSystems2022,wang2025maximal}, and
derive estimates of the capacity of the systems that form an equivalent description of the maximal model. 

While the results in this paper extend those in our previous work \cite{wang2025maximal}, which addressed the maximal hard-core model on the lattice $\Z^2$, they rely on novel combinatorial arguments specific to the triangular lattice. The first new idea, used for the high-activity case, is a new type of contours formed of faces of the lattice avoiding so-called deep holes; this definition is specific to the maximal model and has not previously appeared in the literature. This definition sharpens the count of distinct contours in the lattice and gives rise to an estimate of $\lambda$ that bounds the coexistence region. We remark here that our estimate of $\lambda > 10^6$ for phase coexistence is significantly weaker than the corresponding estimate for the $\Z^2$ lattice in \cite{wang2025maximal} (see Remark~\ref{rem: contour counting} below for additional details). We expect that the estimate can be greatly improved by a better approach to contour-counting than what we have employed. For the low-activity case, we prove the Peierls condition based on a classification of triangles in Delaunay triangulations of MISs in $\A$.

\section{Phase coexistence: The high-activity regime}\label{sec: phase coexistence}
The general approach to establishing a lower bound on $\lambda$ for phase coexistence relies on identifying a ``distinguishing event'' that has markedly different probabilities of occurrence under some pair of distinct boundary conditions. 
This idea goes back to Dobrushin's foundational work \cite{dobrushin1968problem}, and it has been widely adopted in the mathematical literature. A development of this idea enables us to prove the following theorem, which forms the main result
of this section.
\begin{theorem}\label{thm:main}  The maximal hard-core model on $\A$ exhibits phase coexistence for sufficiently large $\lambda$. 
\end{theorem}
We will show that this statement holds for all $\lambda > 10^{6}$. We believe that this bound 
can be improved based on more accurate contour counting in the lattice. An analogous result for $\Z^2$, established
in \cite{wang2025maximal}, holds for all $\lambda>6.3506$ and relies on earlier results of \cite{blanca2019phase} on counting special-type walks in $\Z^2$ (the ``taxi-walk connective constant'').

The triangular lattice can be naturally partitioned into $3$ sublattices, one of which is spanned by the vectors 
$(0,\sqrt 3), (\frac 32, \frac{\sqrt 3}2)$ and the other two are its shifts. This partition defines a 3-coloring of the vertices of $\A$ 
shown in Fig.~\ref{fig:vertex_color}. 

\begin{figure}[H]
    \begin{center}
      \includegraphics[width=0.25\textwidth]{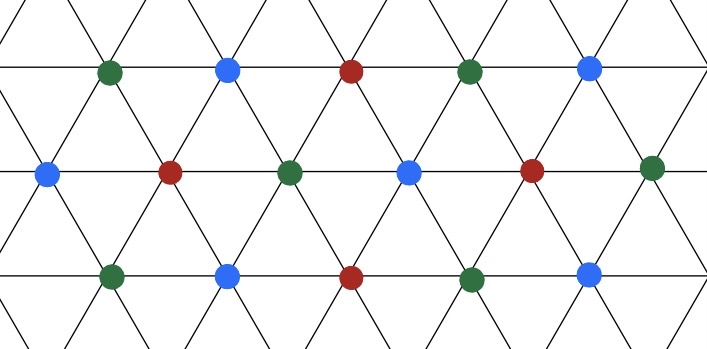}
        \caption{A coloring of the vertices of $\A$}\label{fig:vertex_color}
    \end{center}
    \vspace*{-.15in}
\end{figure}

 Let $\omega^r$, $\omega^g$, and $\omega^b$ be the configurations supported on all red, green, and blue vertices, respectively.
We will always label sites in $\A$ by their coordinates in the basis $(\bfb_1,\bfb_2)=((1,0),(\frac12,\frac{\sqrt3}2))$. Denote by $\Lambda_n = \{-n, \dots, n\}^2 \subset \A$ a rhombus, and let $\Lambda_n^c=\A\backslash\Lambda_n$ be its complement. 
 Consider the set of configurations $\omega\in \cX(\A)$ such that all the vertices of some fixed color, say blue, in $\Lambda_n^c$ are occupied; see \cref{fig: contour} below. 
  We say that these $\omega$'s follow the blue {\em boundary condition}. Denote by $\cJ_n^b$ the set of all configurations in $\cX(\A)$ that agree with $\omega^b$ outside $\Lambda_n$.
The corresponding Gibbs distribution on $\Lambda_n$  is given by
\begin{equation*}
    \mu_n^b (\eta) := \mu_{\Lambda_n}(\eta \mid \omega^b) = \frac{\lambda^{|\eta_{\Lambda_n}|}}{\bfZ_n^b},
\end{equation*}
where $\eta\in \cJ_n^b$ and $\bfZ_n^b = \sum_{\eta \in \cJ_n^b} \lambda^{|\eta_{\Lambda_n}|}$.
The measure $ \mu_n^b $ is supported on $\cJ_n^b$. 
The Gibbs distributions $\mu_n^r, \mu_n^g$ under the red and green boundary conditions
are defined analogously. The Gibbs measures obtained as weak subsequential limits of ${\{\mu_n^b\}}_n, {\{\mu_n^{r}\}}_n, {\{\mu_n^g\}}_n$ for $n\to\infty$ are denoted by $\mu^b, \mu^r, \mu^g$, respectively.

 \begin{figure*}[!t]
  \centering
  \begin{subfigure}[t]{0.3\textwidth}
    \hspace*{.35in}\includegraphics[width=.55\linewidth]{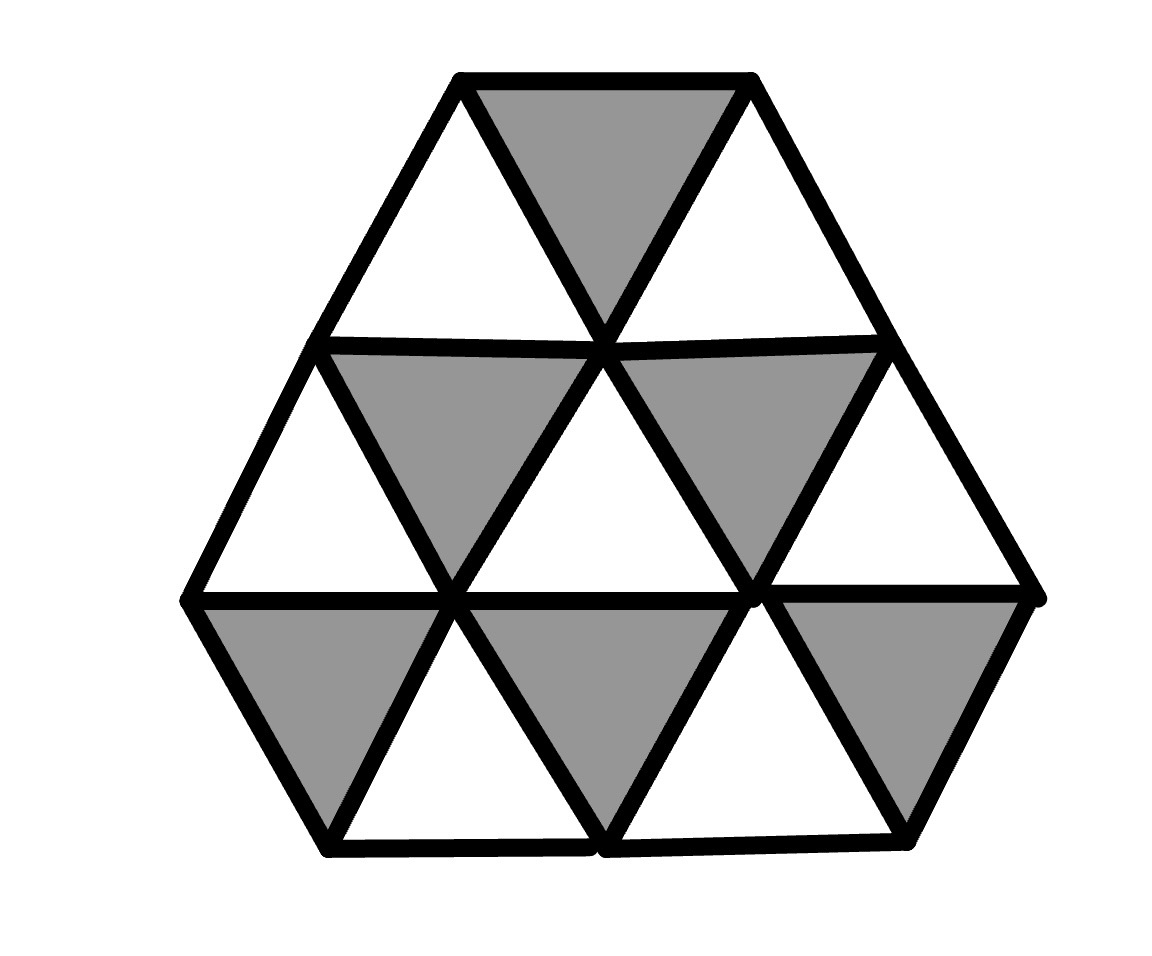}
    \subcaption{The central face is adjacent to $12$ faces and connected to $6$ of them, filled in gray.}\label{fig:face_connectivity}
  \end{subfigure}\hspace*{.5in}
  \begin{subfigure}[t]{0.35\textwidth}
  \captionsetup{width=1.4\textwidth}
    \includegraphics[width=.85\linewidth]{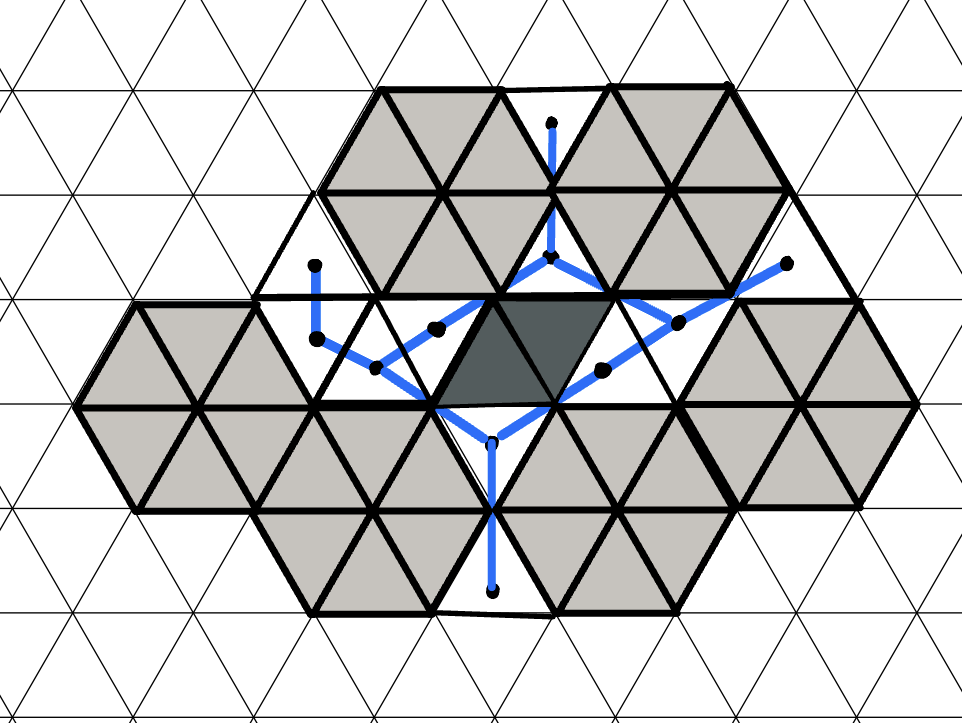}
    \subcaption{A subgraph $G_\gamma\subset G_{\Lambda_n}$ corresponding to a contour (only a part of the subgraph is shown). Filled hexagons are centered on occupied vertices in $\omega$. Deep holes are shown as dark faces. Graphs $G_\gamma$ are used in the counting argument in \cref{lemma: contour counting}.}\label{fig: neighbors}
  \end{subfigure}
  \caption{Connectivity and contours}\label{fig:contour-combined}
\end{figure*}
 
Let $\omega$ be a MIS in $\A$. Below we call triangles in the lattice {\em faces}. A face is \emph{empty} if it is not incident to any of the occupied vertices in $\omega$, and a hexagon is empty if all the faces that form it are empty. Note that the maximality condition rules out empty hexagons in a configuration.
Two faces are called \emph{adjacent} if they share a vertex and are called
\emph{connected} if they either share an edge (we call such a pair \emph{edge-connected}) or they share a single vertex and are not both edge-connected to a common face (we call them \emph{vertex-connected}); see \cref{fig:face_connectivity}.
A configuration, $\omega$, gives rise to a maximal packing of hexagons, which inherit their colors from the center (occupied) sites. A face contained in a blue hexagon will itself be called blue, and similarly for the other colors. Empty faces are uncolored. We also call an empty face a {\em deep hole} if the three faces to which it is edge-connected are empty. Deep holes in a configuration can arise in two ways, as shown in \cref{fig:deep_holes}. 

\begin{figure}[h]
\centering
\includegraphics[width=.4\linewidth]{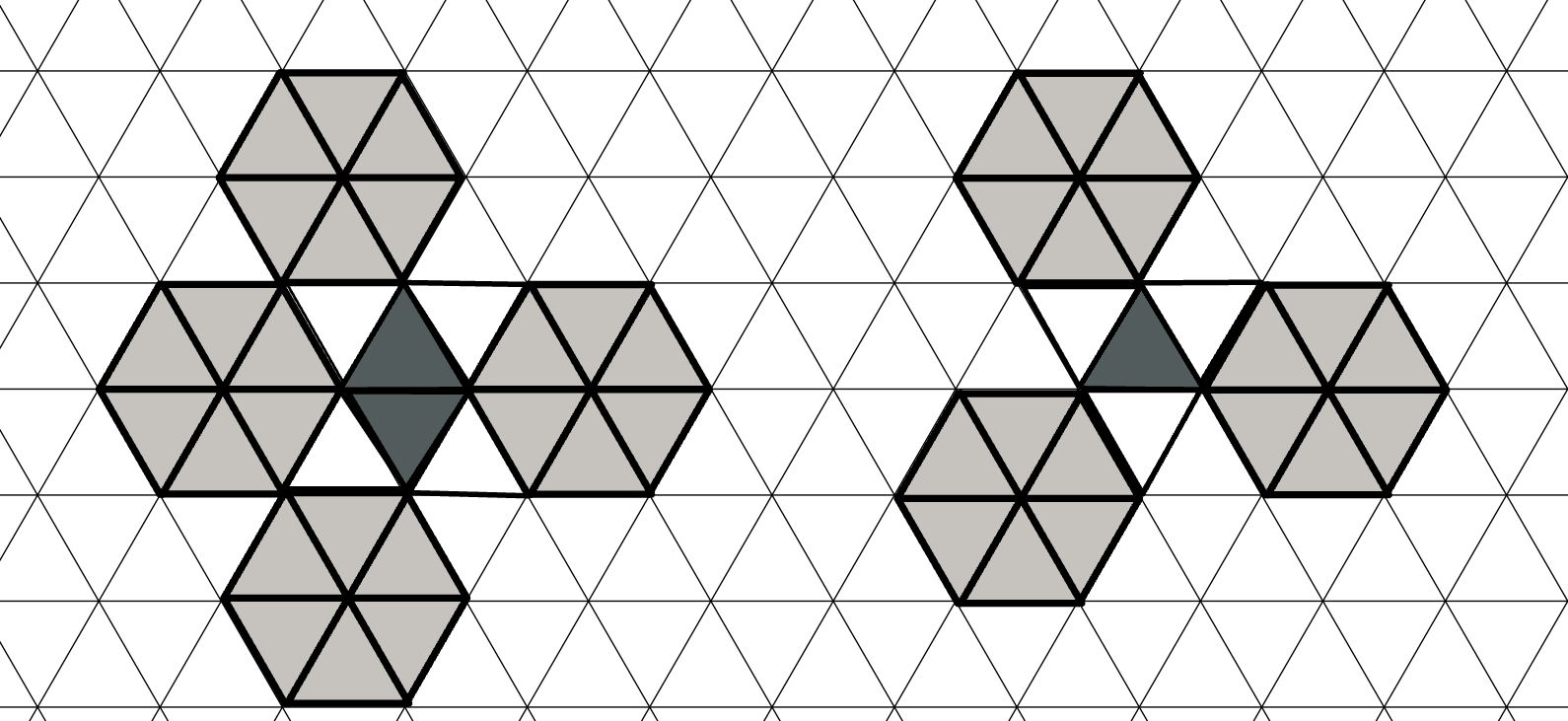}
\caption{Deep holes can be edge-connected or isolated. }\label{fig:deep_holes}
\end{figure}

The faces of $\A$ define a 6-regular graph $G_{\A}$ as follows: the vertices of $G_{\A}$ are the faces of $\A$, and two vertices in $G_{\A}$ are joined by an edge if the corresponding faces of $\A$ are connected in the sense of the above definition. A (simple) path in $G_{\A}$ is a sequence of connected vertices in which no vertex appears twice. A path in $G_{\A}$ is called edge-connected if the vertices in the path are edge-connected as faces of $\A$. The restriction of $G_{\A}$ to $\Lambda_n$ will be denoted $G_{\Lambda_n}$. 

\begin{definition}\label{def:contour}
Given a configuration  $\omega \in \cX(\A)$ subject to a boundary condition, a {\em contour}, $\gamma$, is a maximal connected component in $\Lambda_n$ formed by empty faces other than deep holes. In other words, contours are formed of empty faces edge-connected to a colored face. The restriction of $G_{\A}$ to $\gamma$ will be denoted $G_{\gamma}$; see \cref{fig: neighbors} for an example. The set of faces in $\A$ (vertices of $G_\A$, denoted by $F(\A)$) from which there is an edge-connected path to infinity not intersecting $\gamma$ is called the \emph{exterior} of the contour, denoted $\gamma_\circ$; see \cref{fig: contour}. 
The contour \emph{interior} is defined as $\gamma_\bullet := F(\A) \setminus (\gamma \cup \gamma_\circ\cup\{\text{deep holes connected to $\gamma$}\})$. 
We call the set of faces of $\gamma$ (vertices of $G_{\gamma}$) that are edge-connected to the exterior $\gamma_{\circ}$ the \emph{outer boundary} of $\gamma$, denoted $\partial^{\mathrm{out}}\gamma$. 
\end{definition}

For instance, in \cref{fig: contour}, there are two contours --- one formed by the empty faces inside the cluster of green hexagons, and the other formed by the empty faces other than deep holes located between the blue exterior and the green cluster. If we take $\gamma$ to be the latter contour, the exterior $\gamma_\circ$ is the sea of blue faces extending to infinity in all directions, and $\partial^{\mathrm{out}}\gamma$ consists of all the white faces that are edge-connected to a (blue) face of $\gamma_{\circ}$. If, on the other hand, we take $\gamma$ to be the contour inside the green cluster, then the exterior $\gamma_{\circ}$ consists of the green hexagons and everything beyond. In this case, $\partial^{\mathrm{out}}\gamma$ consists of all the faces of $\gamma$ that are edge-connected to a green face. 
Somewhat different ``thick'' contours were previously considered in \cite{greenberg2010slow}. 

\begin{figure}[h]
{ \begin{subfigure}[t]{0.45\textwidth}
 \hspace*{.0in}{    \includegraphics[width=0.95\textwidth]{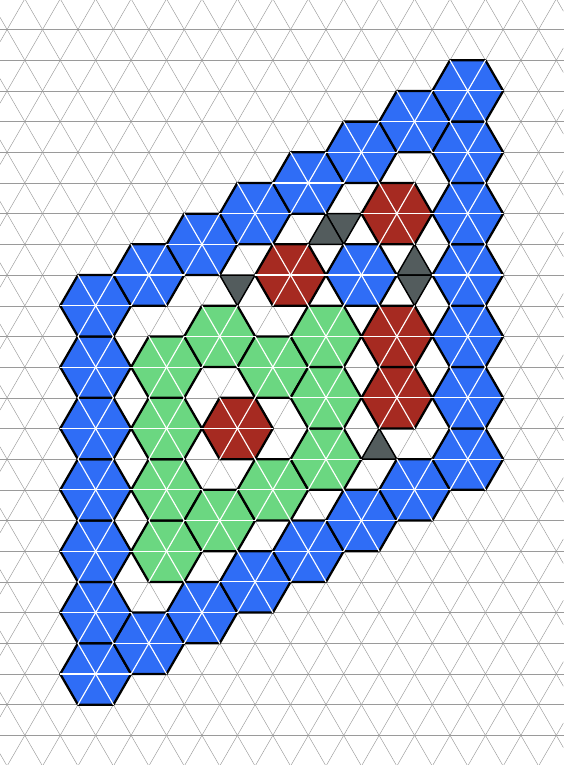}}
    \subcaption{A MIS satisfying the blue boundary condition, giving rise to two contours
  shown by empty triangles inside the rhombus, $\Lambda_n$. Deep holes are shown as dark faces.}\label{fig: contour}
  \end{subfigure}\hspace*{.15in}
 \begin{subfigure}[t]{0.45\textwidth}
 \captionsetup{width=1.15\textwidth,singlelinecheck=false}
\hspace*{.3in}{\includegraphics[width=.95\linewidth]{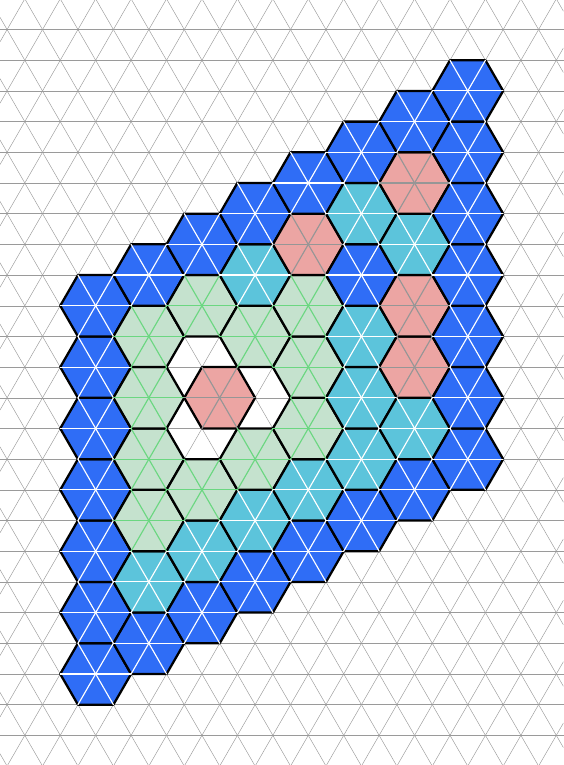}}
    \subcaption{The configuration in \cref{fig: contour} upon applying the shift, $\phi$. The green island shifted NW, becoming blue. The shift includes the inner contour and its interior. All the hexagons except the red one inside the inner, green island become blue after the shift. They are shown as lightly filled according to the original color. Newly added blue-colored hexagons are shown as light blue.}\label{fig: shift}
  \end{subfigure}
  }\caption{Illustrating the contour erasing procedure.}
\end{figure}

The collection of contours defined by $\omega$ is denoted by $\Gamma(\omega)$,
and $\Gamma := \cup_{\omega \in \cX(\A)} \Gamma(\omega)$.
Properties of the contours are given in the next lemma.
\begin{lemma} \label{lemma: interior} Let $\omega_{\Lambda_n}$ be a configuration and $\gamma\in\Gamma(\omega)$ be a contour
associated with it. 
 \\\indent    (a) The exterior $\gamma_\circ$ forms a unique infinite connected component.
 \\\indent  (b)  A face in $\gamma$ has either 2 or 3 neighbors in $\gamma$ (\cref{fig: neighbors}).
 \\\indent  (c) The faces of $\partial^{\mathrm{out}} \gamma$ form a simple cycle in $G_{\gamma}$.
 \\\indent  (d)  The faces of $\gamma_\circ$ connected to $\gamma$ have the same color. 
\end{lemma}

\begin{proof} Part (a) follows by definition. For the other parts of the proof, we note that the only way 
to prevent a contour from extending in a particular direction is to have a group of edge-connected hexagons blocking that direction; such hexagons necessarily have the same color.
Turning to part~(b), we note that  a face in $\gamma$ can have degree one if it is edge- or vertex-connected to a single face.
In either case, it is not possible to place same-color hexagons to block all other connections. Examples of degrees 2 and 3 are given in \cref{fig: neighbors}. Degree 4 is not possible by a case study detailed in the proof of Lemma~\ref{lemma: contour counting} below. Finally, if 5 out of the 6 faces connected to an empty face $e$ are themselves empty, then either 3 of them are vertex-connected to $e$, or 3 are edge-connected to it. In either case, this forces an empty hexagon as illustrated
in \cref{fig: HighDegrees}.

Let us prove part~(c) for an arbitrary contour contained in $\Lambda_n$; part (d) will follow as a consequence. We have assumed that $\Lambda_n$ satisfies some (say, blue) boundary condition. Consider the subgraph of $G_{\A}$ induced by all the blue faces, and consider its connected component that contains the blue exterior of $\Lambda_n$. This is a connected region, $\mathbb{B}$, formed by blue hexagons that may extend inside $\Lambda_n$. The complement of $\mathbb{B}$ in $\mathbb A$ is a disjoint union of simply connected domains contained within $\Lambda_n$. 
In each of these domains $\mathbb{D}$, the empty faces that are edge-connected to faces in $\mathbb{B}$ form the outer boundary, $\partial^{\mathrm{out}} \gamma^{\mathbb{D}}$, of a contour $\gamma^{\mathbb{D}}$. The faces in $\partial^{\mathrm{out}} \gamma^{\mathbb{D}}$ follow the edges on the boundary of the domain $\mathbb{D}$, so they are arranged in the form of a closed circuit. Moreover, successive faces in this arrangement are either edge-connected or vertex-connected. Thus, the faces in $\partial^{\mathrm{out}} \gamma^{\mathbb{D}}$ form a cycle in $G_{\gamma^{\mathbb{D}}}$. 
This cycle is, in fact, simple since no face of $\partial^{\mathrm{out}} \gamma^{\mathbb{D}}$ can be edge-connected to two distinct blue faces in $\mathbb{B}$. Thus, part~(c) of the lemma holds for $\gamma^{\mathbb{D}}$. Moreover, part~(d) trivially holds for $\gamma^{\mathbb{D}}$, since $\gamma^{\mathbb{D}}_{\circ} \subset \mathbb{B}$.

\begin{figure}[H]
    \begin{center}
      \includegraphics[width=0.35\textwidth]{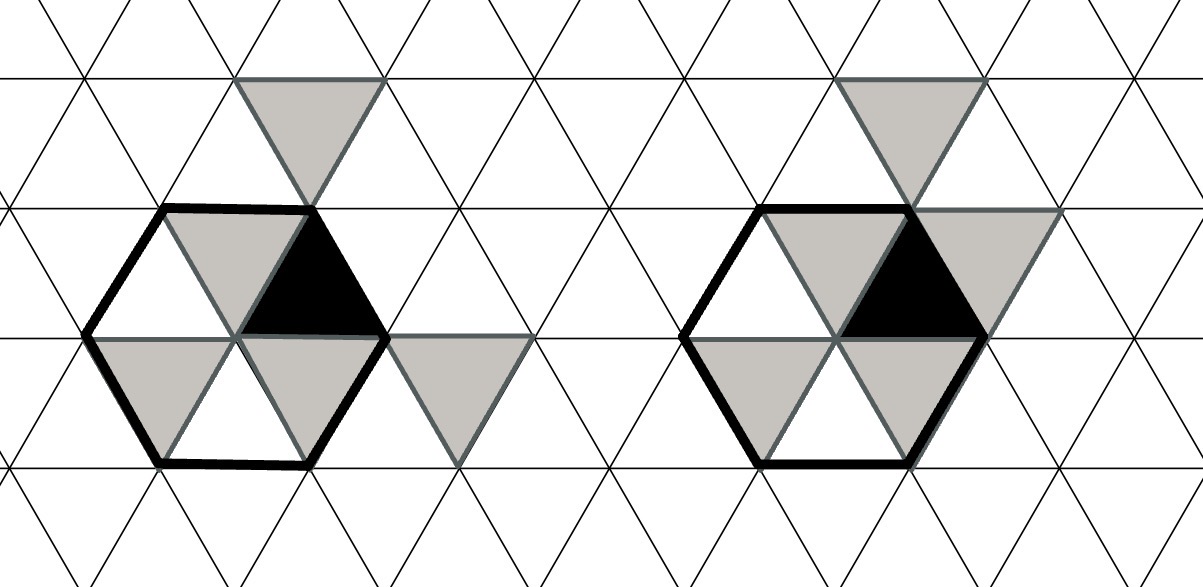}
        \caption{Two possible cases with an empty triangle (filled in black), five of whose neighbor triangles (filled in gray) are also empty. Each of these arrangements forces an empty hexagon, so they cannot arise in a MIS. }\label{fig: HighDegrees}
    \end{center}
    \vspace*{-.15in}
\end{figure}

Consider now any other contour $\gamma$, which necessarily lies in the interior of some $\gamma^{\mathbb{D}}$ considered above. The contour $\gamma$ is, by definition, disconnected from all other contours in its exterior (including $\gamma^{\mathbb{D}}$). This is possible only if $\gamma$ is surrounded by a ``boundary wall'' of edge-connected hexagons of the same color. (See, for example, the inner contour in \cref{fig: contour} that is surrounded by a wall of green hexagons.) The faces of $\gamma$ that are edge-connected to this boundary wall constitute $\partial^{\mathrm{out}} \gamma$. Part~(c) of the lemma now follows from an argument analogous to that given above for $\gamma^{\mathbb{D}}$. Finally, part~(d) is due to the fact that any face of $\gamma_{\circ}$ that is connected to a face of $\gamma$ belongs to the monochromatic boundary wall.
\end{proof}

Below we say that $\gamma$ is a red (blue, green) contour if the color identified in \cref{lemma: interior}(d) is red (blue, green, respectively).

 For a fixed number $m<n$, define $\Lambda_m^E = \{-m, \dots, m+1\} \times \{-m, \dots, m\}$ ($\Lambda_m^W = \{-m-1, \dots, m\} \times \{-m, \dots, m\}$) to be the collection of vertices that are contained in $\Lambda_m$ and the right (left) neighbors of $\Lambda_m$.

Given a configuration $\eta \in \cX(\A)$, we say that 
\begin{itemize}
    \item $\eta$ is {$m$-red} if all red vertices in $\Lambda_m^{E}$ are occupied.
    \item $\eta$ is {$m$-green} if all green vertices in $\Lambda_m^{W}$ are occupied.
    \item $\eta$ is $m$-homogeneous if it is either $m$-green or $m$-red.
\end{itemize}

To prove Theorem~\ref{thm:main}, we seek to construct an event for which the conditional probabilities differ under the two limiting measures. Specifically, let $\sR_m$ denote the event that a configuration is $m$-red, and $\sG_m$ the event that it is $m$-green. 
We will show that for all $n > m$ and sufficiently large $m$, 
\begin{align}\label{eq:red_boundary}
    \mu_n^r (\sG_m \mid \sG_m \cup \sR_m)  &\le \frac{1}{3},\\
\intertext{and}
 \label{eq:green_boundary}
    \mu_n^g (\sG_m \mid \sG_m \cup \sR_m) &\ge \frac{2}{3}.
\end{align}
These inequalities show that, under the condition that a configuration is $m$-homogeneous, the red boundary condition $\omega^r$ implies that it is unlikely to be $m$-green, and vice versa. Since the limiting Gibbs measures depend on the boundary condition, \cref{eq:red_boundary,eq:green_boundary} together imply the claimed phase coexistence. 

Let $\cB_{m,n}^r$ ($\cA_{m,n}^r$) be the set of $m$-green (resp., $m$-homogeneous) configurations that are consistent with the red boundary condition $\omega^r$ outside $\Lambda_n$.
To prove~\eqref{eq:red_boundary}, we will construct a mapping $\phi : \cB_{m,n}^r \to \cA_{m,n}^r \setminus \cB_{m,n}^r$ such that for a fixed large $m$ and $\tau \in \cA_{m,n}^r \setminus \cB_{m,n}^r$,
\begin{equation}\label{eq:shift_bound}
    \sum_{\substack{\eta \in \cB_{m,n}^r: \\ \phi(\eta) = \tau}} \lambda^{|\eta_{\Lambda_n}| - |\tau_{\Lambda_n}|} \le \frac{1}{3}
\end{equation}
holds for all  $\tau \in \cA_{m,n}^r \setminus \cB_{m,n}^r$. 
Once \eqref{eq:shift_bound} is proved, we find 
\begin{align}
        \mu_n^r (\sG_m \mid\sG_m \cup \sR_m) &= \frac{\sum_{\eta \in \cB_{m,n}^r} \lambda^{|\eta_{\Lambda_n}|}}{\sum_{\eta \in \cA_{m,n}^r} \lambda^{|\eta_{\Lambda_n}|}} \notag\\
        &= \frac{\sum_{\eta \in \cB_{m,n}^r} \lambda^{|\phi(\eta)_{\Lambda_n}|} \cdot \lambda^{|\eta_{\Lambda_n}| - |\phi(\eta)_{\Lambda_n}|}}{\sum_{\eta \in \cA_{m,n}^r} \lambda^{|\eta_{\Lambda_n}|}} \notag\\
        & = \frac{\sum_{\tau \in \cA_{m,n}^r \setminus \cB_{m,n}^r} \lambda^{|\tau_{\Lambda_n}|}  \sum_{\substack{\eta \in \cB_{m,n}^r \\ \phi(\eta) = \tau}} \lambda^{|\eta_{\Lambda_n}| - |\tau_{\Lambda_n}|} }{\sum_{\eta \in \cA_{m,n}^r} \lambda^{|\eta_{\Lambda_n}|}} \le \frac{1}{3}, \label{eq: mu^e}
\end{align}
which establishes \cref{eq:red_boundary}.
To construct the mapping $\phi$, we begin by defining a contour elimination operation. 

An \emph{island} of $\gamma$ is an edge-connected component of $\gamma_\bullet$.
The hexagons of the island that share an edge with faces of $\gamma$ are of the same color. We call this set of hexagons the ``boundary wall'' and we say that the island is blue (red, green) if its boundary wall is blue (resp., red or green). Note that there can be further contours contained in the island (as shown in \cref{fig: contour}); however, the shifts that we discuss below apply to the entire island, including possibly those contours.

\textbf{Contour elimination operation $\psi : \cX(\A) \times \Gamma  \to \cX(\A)$.} 
Let  $\gamma \in \Gamma(\omega)$ be a contour of some configuration $\omega \in \cX(\A)$. 
Define the operation $\psi(\omega, \gamma)$ as follows.
\begin{enumerate}
    \item Shift the islands: 
    \begin{itemize}
    \item If $\gamma$ is red, shift all blue islands by one unit in the northwest (NW) direction (i.e., add $\bfb_2-\bfb_1$ to each site) and shift all green islands by one unit in the northeast (NE) direction.
    \item If $\gamma$ is green, shift all blue islands NE by one unit and shift all red islands NW by one unit.
    \item If $\gamma$ is blue, shift all red islands one unit NE and all green islands one unit NW.
    \end{itemize}
Islands of the same color as the color of $\gamma$ that may occur in $\gamma_\bullet$ are not shifted. This operation shifts the islands of the remaining two colors in the directions that ensure that the shifted sites do not compete for the same spot (if they did, then before the shift, the sites of these two colors would be adjacent). Denote the
set of shifted sites by $T(\gamma_\bullet)$.
\item Change the color of the shifted sites to agree with the color of the contour $\gamma$.  In \cref{fig: shift}, the colors of the (lightly filled) green and red hexagons change to blue and green, respectively. 
    \item Add new occupied sites to the shifted configuration such that the resulting configuration is in $\cX(\A)$. These vertices will have the color of $\gamma$, thus erasing the contour. The newly added hexagons are shown in light blue in
 \cref{fig: shift}.   
\end{enumerate}
Note that the directions of the shift are not unique; we only need to make sure that collisions do not arise as a result.

\begin{lemma}\label{lemma:gamma/6}
The operation $\psi(\cdot, \gamma)$ is injective and for every $n$ such that $\gamma \subset \Lambda_n$, we have
\begin{equation}\label{eq:gamma/6}
|\psi(\omega, \gamma)_{\Lambda_n}| - |\omega_{\Lambda_n}| = \frac{|\gamma_\vartriangle|}{6} \ge \frac{|\gamma|}{6},
\end{equation}
where $\gamma_\vartriangle=\gamma\cup\{\text{deep holes connected to $\gamma$}\}$.
\end{lemma}
\begin{proof}
    Let $\eta = \psi(\omega, \gamma)$. 
    To argue that $\psi$ is injective given $\gamma$, we simply note that one can reconstruct $\omega$ from $\eta$ and $\gamma$. Knowing $\gamma$, we also know the boundary of the contour and the location of the deep holes. To reconstruct $\omega$, we first remove the hexagons added in Step 3. They are identifiable because they cannot be a part of $T(\gamma_\bullet)$ because of the direction of the shift. Next, we identify $T(\gamma_\bullet)$ and undo the shift, recovering $\omega$.
    
    To show \cref{eq:gamma/6}, we first observe that the vertices added in Step 3 are contained in 
    $\tilde{\gamma} :=  ( \gamma_\vartriangle \cup \gamma_\bullet ) \setminus T(\gamma_\bullet)$. 
    Indeed, the shift operation only moves vertices (filled hexagons) in $\gamma_\bullet$, and one cannot occupy any extra vertex (or put an extra filled hexagon) in $\gamma_\circ$ or $T(\gamma_\bullet)$. 
    Note that all faces outside $\Tilde{\gamma}$ that are \emph{adjacent} to $\Tilde{\gamma}$ have the same color after the shift; 
    thus, the only way to make $\psi(\omega,\gamma)$ maximal is to cover all faces in $\Tilde{\gamma}$ by adding vertices (hexagons) of the same color. 
\end{proof}
We now proceed to establish~\eqref{eq:shift_bound}. 
Since $\eta \in \cB_{m,n}^r$ is $m$-green, there exists a red contour that separates the green homogeneous region containing $\Lambda_m^W$ and the red boundary.
Let $\gamma$ be that contour and define $\phi(\omega) := \psi(\omega, \gamma)$.
By~\eqref{eq:gamma/6} we have
\begin{equation} \label{eq:eta_tau}
        \begin{split}
            \sum_{\substack{\eta \in \cB_{m,n}^r: \\ \phi(\eta) = \tau}} \lambda^{|\eta_{\Lambda_n}| - |\tau_{\Lambda_n}|} \le \sum_{\substack{\eta \in \cB_{m,n}^r: \\ \phi(\eta) = \tau}} \lambda^{-\frac{|\gamma(\eta)|}{6}}  \le \sum_{\substack{\gamma \in \Gamma: \\ \Lambda_m^W \subset \gamma_\bullet}}\lambda^{-\frac{|\gamma|}{6}},
        \end{split}
\end{equation}
where the first inequality is true for all $\lambda>1$ and the second one follows since we can reconstruct $\eta$ given $\gamma \in \Gamma$ and $\tau$.
Now let $\Gamma^m_l$ be the collection of all contours of size $l$ (with $l$ faces) that contain $\Lambda_m^W$. We have
\begin{equation}\label{eq:series}
    \sum_{\gamma \in \Gamma: \Lambda_m^W \subset \gamma_\bullet}\lambda^{-\frac{|\gamma|}{6}} = \sum_{l \ge 4m}  \frac{|\Gamma^m_l|}{\lambda^{l/6}}.
\end{equation}
\begin{lemma}\label{lemma: contour counting} 
$|\Gamma_l^m| = O(l^2 10^{l})$.
\end{lemma}
\begin{proof} 
 A subgraph is associated with any contour, as in \cref{fig: neighbors}, and distinct contours give rise to distinct subgraphs. We will estimate the number of subgraphs of $G_{\Lambda_n}$ with degrees 2 or 3 (of course, not all of them give rise to valid contours). Fix a vertex $v\in V(G_{\Lambda_n})$ to be the starting point of the contour $\gamma$. Since $\gamma$ surrounds the origin, there are at most $O(l^2)$ possible locations for $v$.
 
A priori, there are at most $\binom62+\binom63+\binom64=50$ ways to choose the neighbors of $v$. From each of the neighbors, we can choose a further set of neighbors in at most $\binom51 + \binom52 + \binom53 = 25$ ways. Carrying on recursively, this gives an upper bound of $50 \times 25^{l-2} = O(25^l)$ for the number of connected $l$-vertex subgraphs of $G_{\Lambda_n}$ starting with $v$. Accounting for the $O(l^2)$ possible choices of $v$, this gives us a (crude) upper bound of $O(l^2 25^l)$ for the number of possible contours in $\Gamma_l^m$.

To reduce the base from $25$ to $10$, we count the subgraphs more carefully. The options for choosing the neighbors in a construction step other than the first
one are shown in \cref{fig: patterns}. We start with a vertex, $w$, shown by a black-filled face. It is connected
to (at least) one face (vertex) in the graph, $u$, shown as a dash-filled triangle. $u$ and $w$ can be edge- or vertex-adjacent, and we examine both options as indicated in the captions, choosing the larger count for the overall estimate. 

First, we argue that the graph $G_\gamma$ contains no vertices of degree 4, in other words, that $w$ cannot have 3 neighbors in addition to $u$. Vertices $w$ and $u$ are either vertex- or edge-connected. For the vertex-connected case, all the possibilities of placing the three neighbors are shown in \cref{fig:3v}. The first 5 configurations give rise to a deep hole and therefore are not feasible, and the second five force an empty hexagon (in addition, some of them also give rise to a deep hole).
Similarly, all the configurations for the edge-connected case can be ruled out; see \cref{fig:3e}. In summary, no vertex of $G_\gamma$ can have degree 4. 

All the possibilities for a degree-3 vertex appear in \cref{fig:2v,fig:2e}. The vertex-connected (edge-connected) case gives rise to 7 (resp., 6) distinct feasible configurations, shown as a part of the contour. The last 3 configurations out of 10 in \cref{fig:2v} and the last 4 in \cref{fig:2e} either yield a deep hole or an empty hexagon, and thus are infeasible. A similar analysis is performed for degree 2, summarized in \cref{fig:1v,fig:1e}.

The number of feasible patterns is thus not more than 7 for degree 3 and $3$ for degree 2, yielding $10$ as an upper bound. 
\end{proof}

\begin{figure}[p]
  \centering
  \captionsetup{width=\textwidth}
\hspace*{.4in}\resizebox{\linewidth}{!}{
    \begin{minipage}{\linewidth}
  \begin{subfigure}[t]{0.4\textwidth}
    \includegraphics[width=\linewidth]{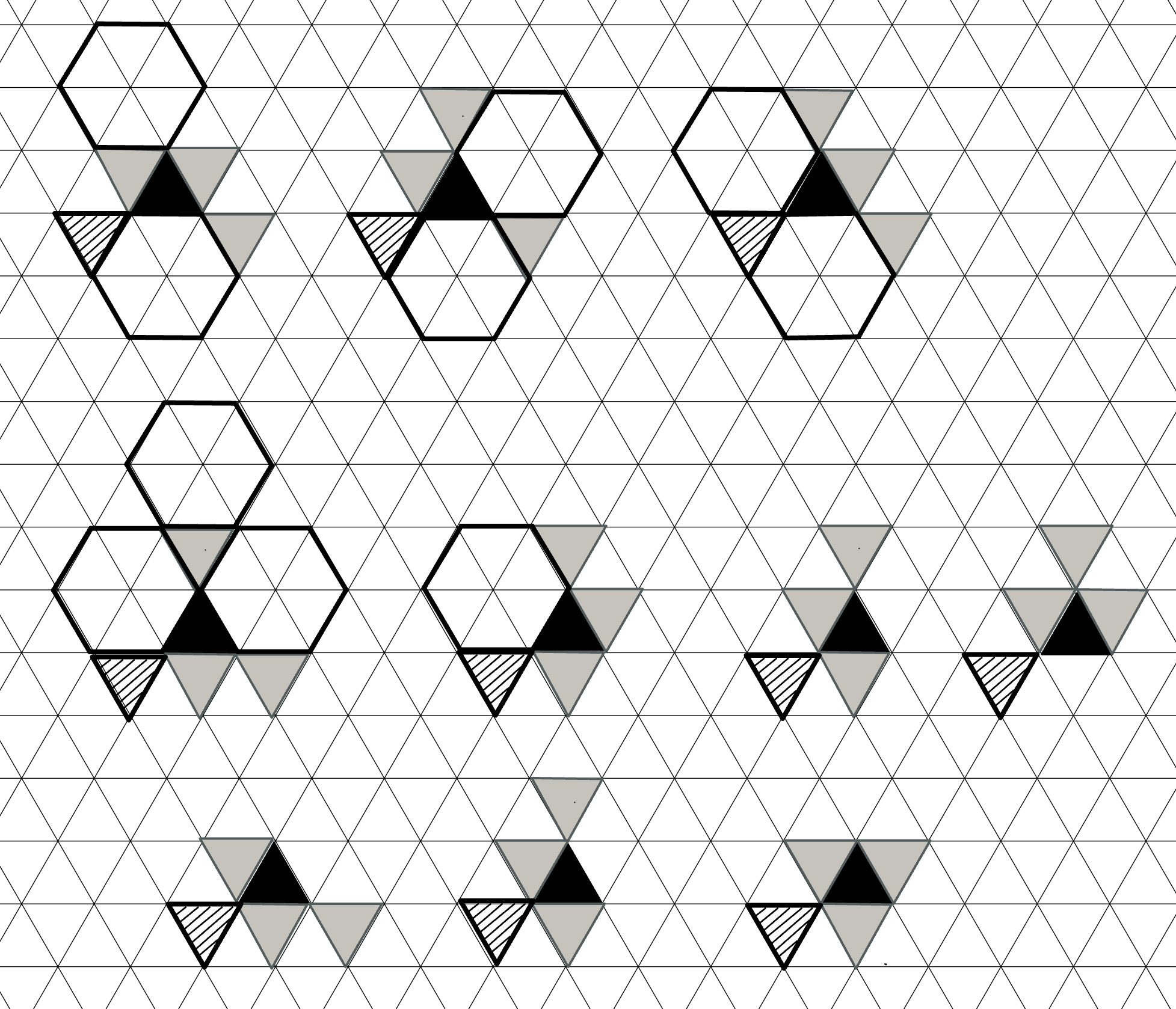}
    \subcaption{4 neighbors, vertex-connected}\label{fig:3v}
  \end{subfigure}\hspace*{.7in}
  \begin{subfigure}[t]{0.4\textwidth}
    \includegraphics[width=\linewidth]{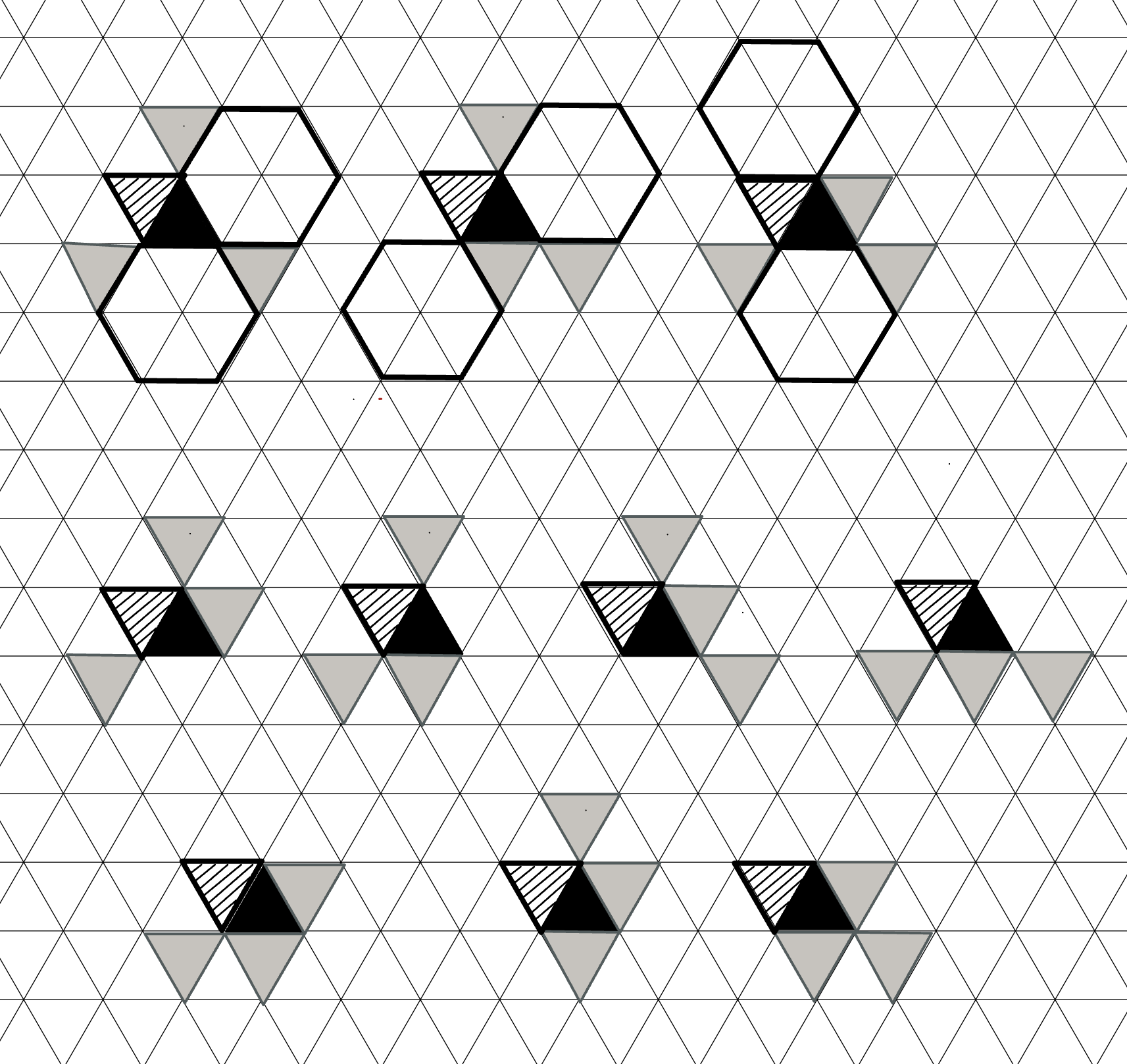}
    \subcaption{4 neighbors, edge-connected}\label{fig:3e}
  \end{subfigure}
 \vspace{0.5em}
 
  \begin{subfigure}[t]{0.4\textwidth}
    \includegraphics[width=\linewidth]{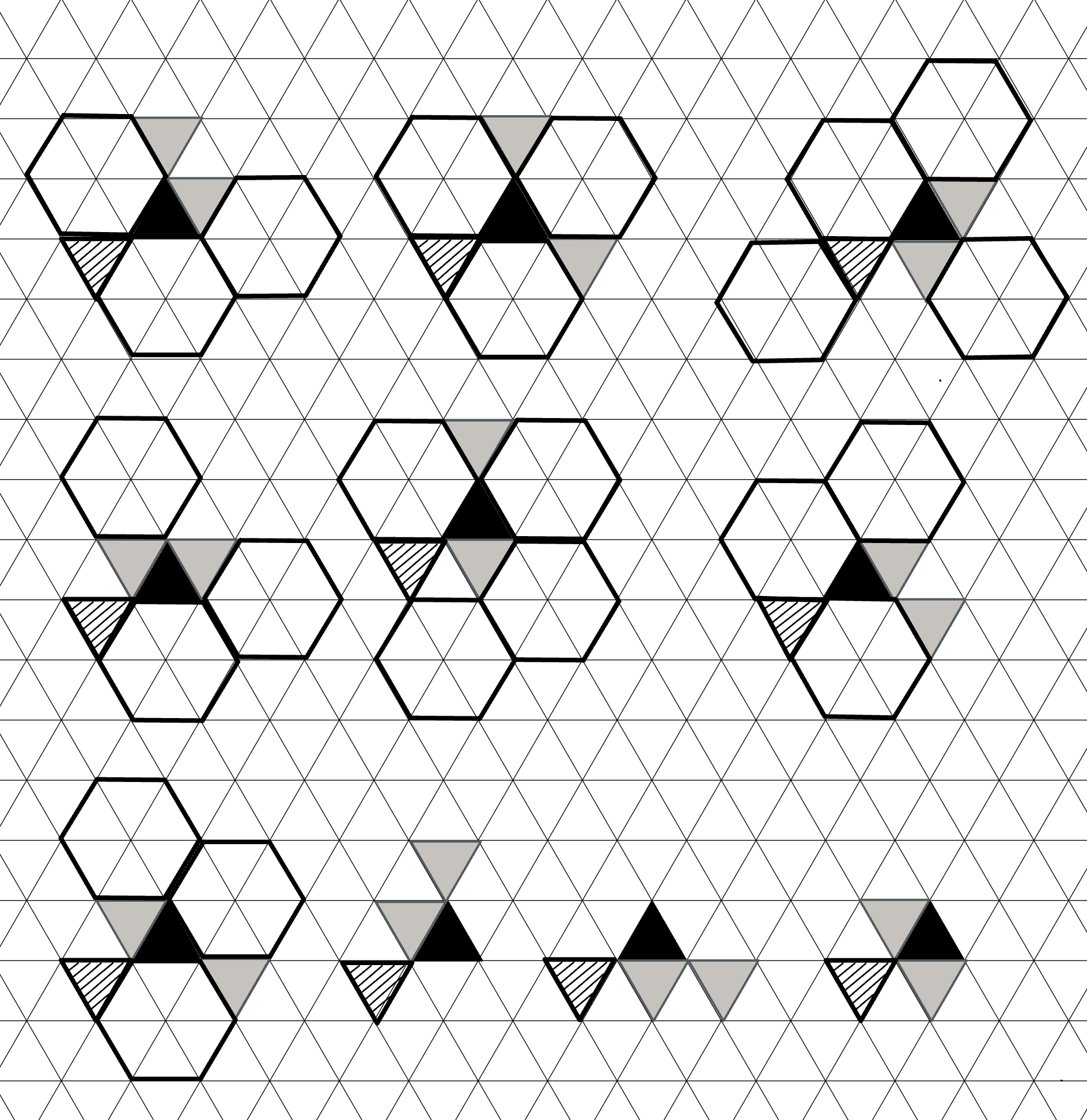}
    \subcaption{3 neighbors, vertex-connected}\label{fig:2v}
  \end{subfigure}\hspace*{.7in}
  \begin{subfigure}[t]{0.4\textwidth}
    \includegraphics[width=\linewidth]{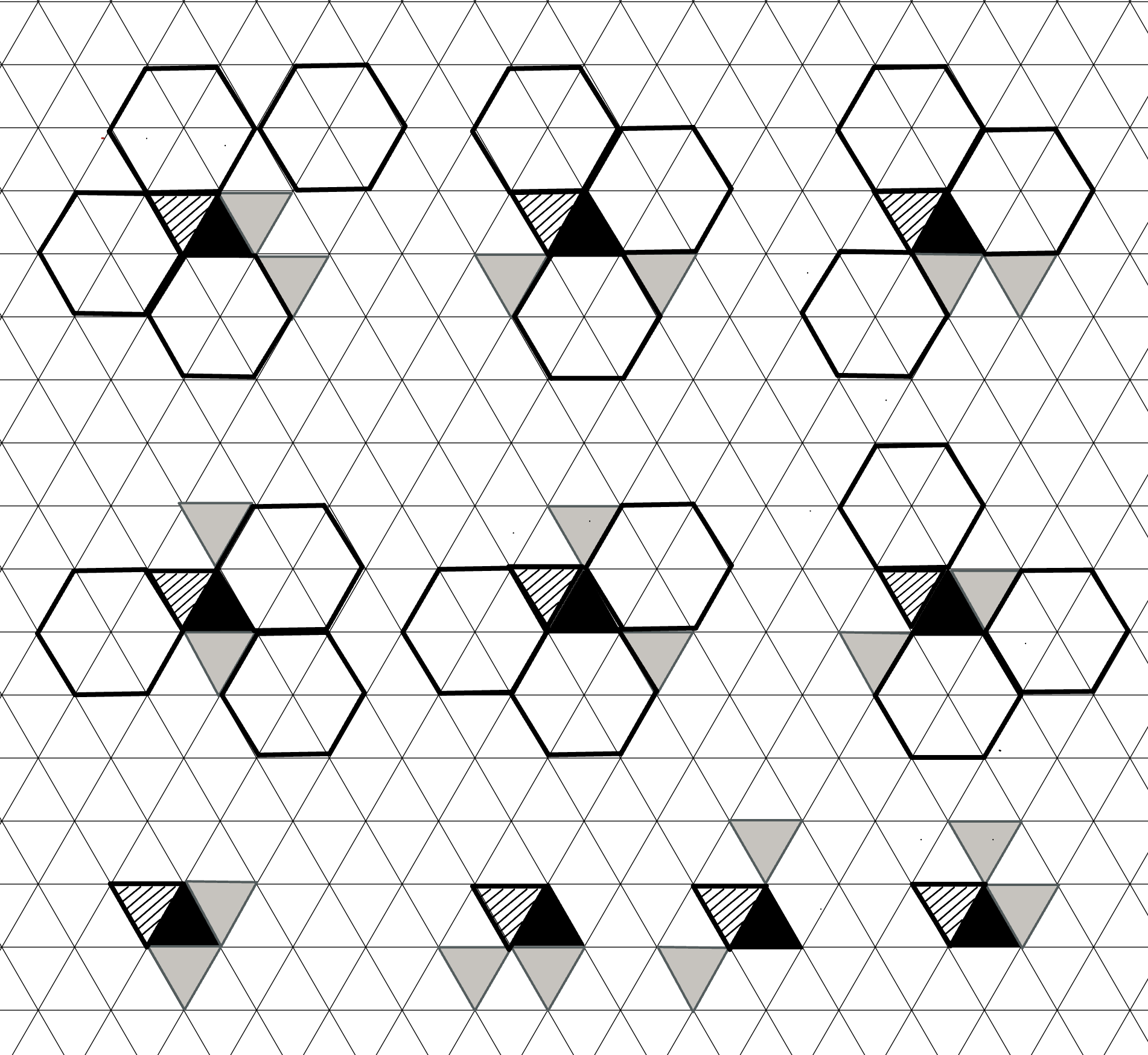}
    \subcaption{3 neighbors, edge-connected}\label{fig:2e}
  \end{subfigure}
 
\vspace{0.5em}

  \begin{subfigure}[t]{0.42\textwidth}
    \includegraphics[width=\linewidth]{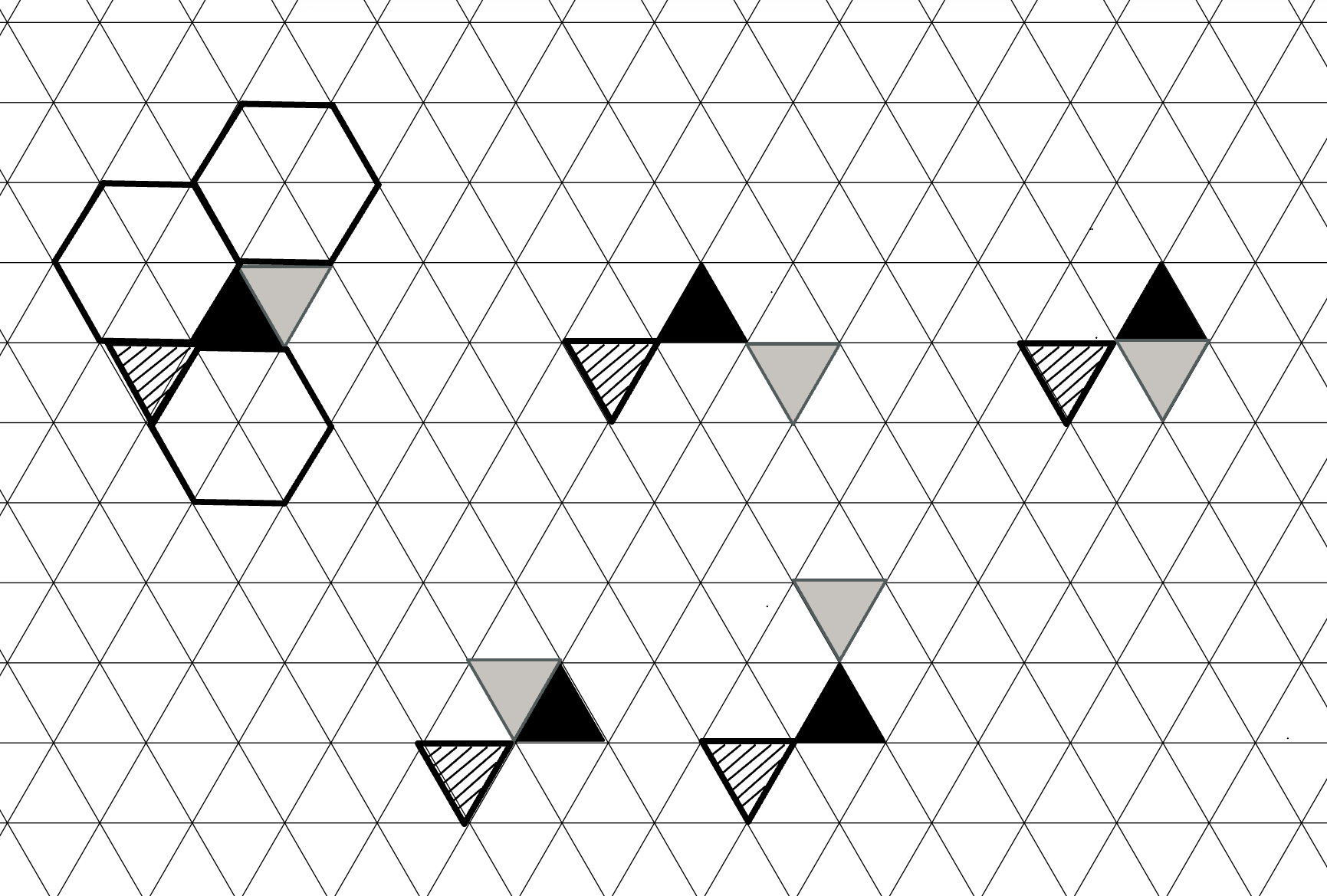}
    \subcaption{2 neighbors, vertex-connected}\label{fig:1v}
  \end{subfigure}\hspace*{.7in}
  \begin{subfigure}[t]{0.36\textwidth}
    \includegraphics[width=\linewidth]{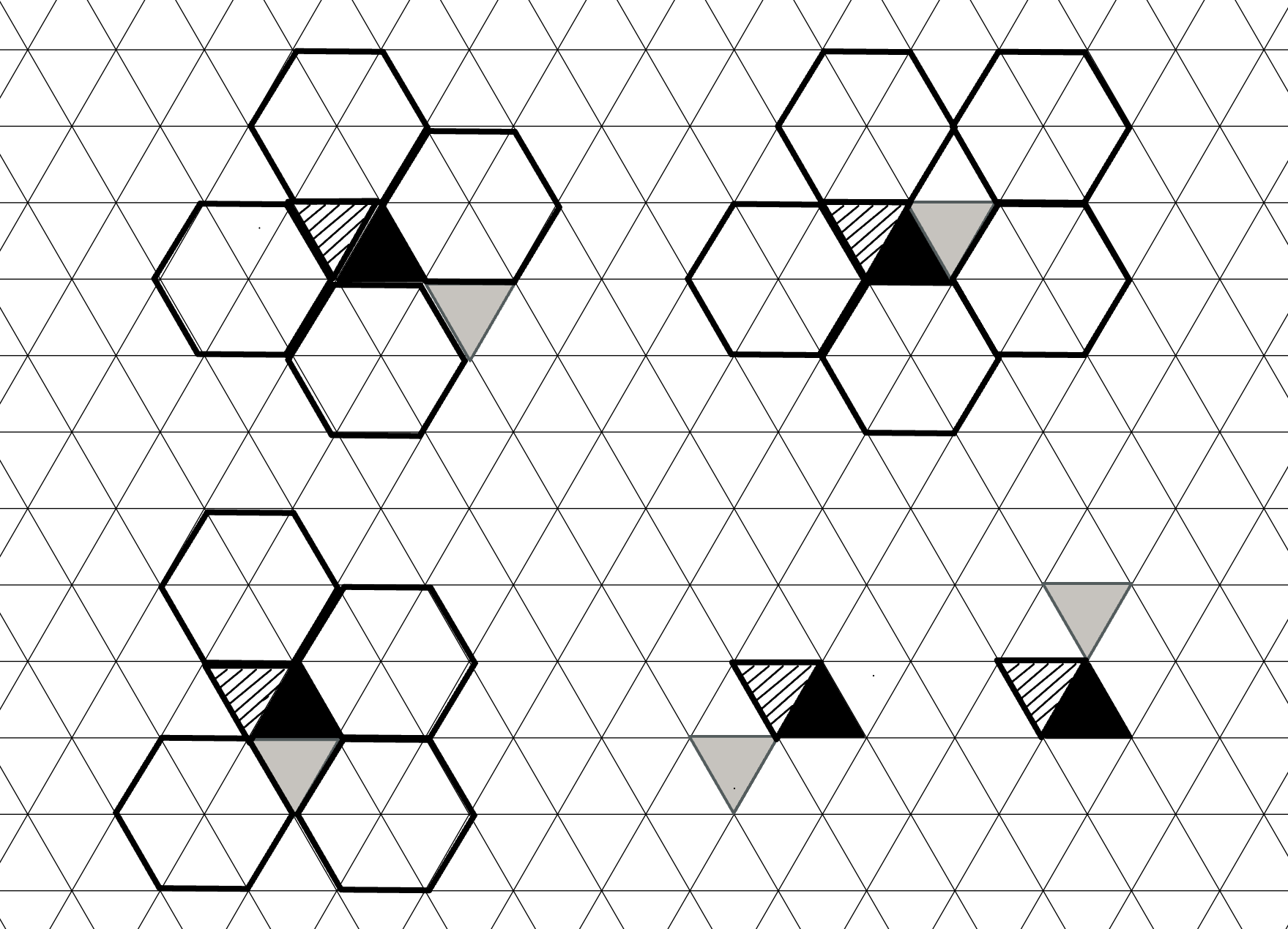}
    \subcaption{2 neighbors, edge-connected}\label{fig:1e}
  \end{subfigure}
  
  \end{minipage}
  }
 \caption{A single step of contour construction in the proof of \cref{lemma: contour counting}. Suppose that the black-filled face is adjacent to the dash-filled face and both are part of the contour. 
 We are estimating the number of ways the contour can continue from the black face. \cref{fig:3v,fig:3e} address the case of three neighbors not counting the dash-filled one, and \cref{fig:2v,fig:2e} and \cref{fig:1v,fig:1e} do the same for two and one neighbor, respectively. The potential neighbors are shown as lightly filled faces. Further explanations appear in the proof of the lemma. }
 \label{fig: patterns}
\end{figure}

Thus, if $\lambda>10^6$, the series in \eqref{eq:series} converges, and its sum can be made smaller than $\frac13$ by choosing $m$ large enough. 
This proves inequality~\ref{eq:red_boundary}. The proof of~\eqref{eq:green_boundary} proceeds analogously by interchanging the roles of red and green, resulting in \(
\mu_n^g(\sR_m \mid \sG_m \cup \sR_m) \le \frac{1}{3}\). This completes the proof of Theorem~\ref{thm:main}.

We note that it may be possible to reduce the exponent base even further if we account for the properties of the graphs
$G_{\gamma}$, summarized as follows.
\begin{lemma} Let $\gamma$ be a contour. The graph $G_{\gamma}$ is bipartite, planar, 2-connected, with degrees 2 or 3.
\end{lemma}
\begin{proof} 
 The statement about degrees follows by combining \cref{lemma: interior,lemma: contour counting}. To show that $G_{\gamma}$ is bipartite, note that faces pointing up can connect only to faces pointing down (cf.\ \cref{fig:face_connectivity}). 

Let us prove that $G_\gamma$ admits a planar embedding in $\R^2$. Place each vertex in the geometric center of the face to which it corresponds and connect the adjacent vertices with straight-line segments. 
Let $U,V$ be two faces in $\gamma$ and let $u,v$ be the corresponding vertices of $G_\gamma$. Vertices $u$ and $v$ are adjacent in $G_\gamma$ if the faces $U$ and $V$ are edge- or vertex-adjacent. If they are edge-connected, then the edge $uv$ cannot intersect other edges of $G_\gamma$: the only edges of $G_\gamma$ that can be present are incident to $u$ or $v$ and do not intersect $uv$. Now suppose that $U$ and $V$ are vertex-connected. The edge $uv$ can intersect another edge of $G_\gamma$ only if there is another vertex-connected pair of faces, $U'$ and $V'$, that shares the vertex of $\A$ common to $U$ and $V$. However, such empty pairs $U,V$ and $U',V'$ cannot appear in the contour together because this would force an empty hexagon.

\hspace*{-.3in}\begin{figure}[h]
  \centering
\hspace*{-.3in} { \begin{subfigure}[t]{0.15\textwidth}
{    \includegraphics[width=.9\linewidth]{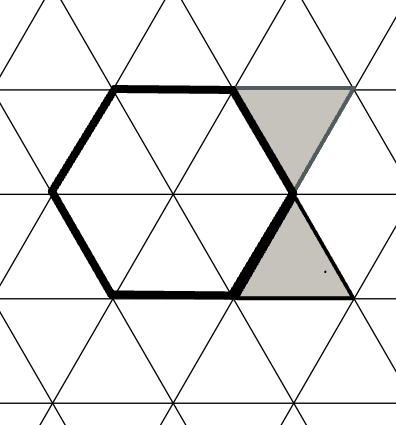}}
   \captionsetup{width=1.25\textwidth} \subcaption{Vertex-connected}\label{fig:VC}
  \end{subfigure}\hspace*{.2in}
 \begin{subfigure}[t]{0.16\textwidth}
{\includegraphics[width=.85\linewidth]{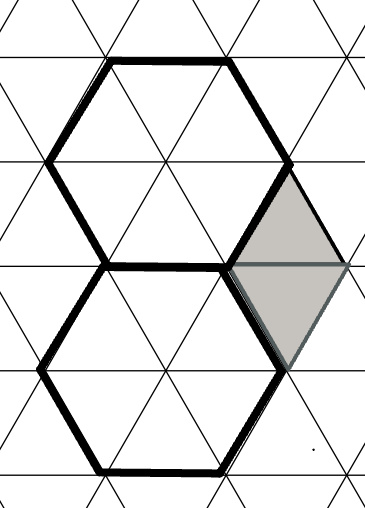}}
\captionsetup{width=1.2\textwidth}    \subcaption{Edge-connected, I}\label{fig:EC1}
  \end{subfigure}\hspace*{.2in}
  \begin{subfigure}[t]{0.17\textwidth}
    \includegraphics[width=\linewidth]{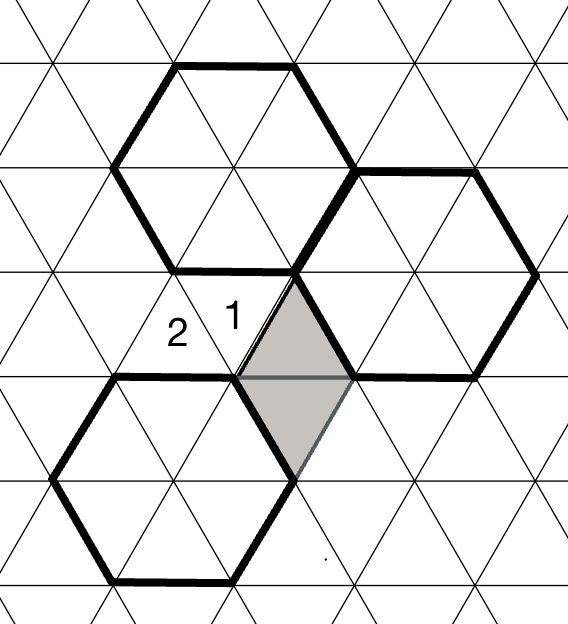}
  \captionsetup{width=1.2\textwidth}  \subcaption{Edge-connected, II}\label{fig:EC2}
  \end{subfigure}\hspace*{.2in}
 \begin{subfigure}[t]{0.17\textwidth}
 \captionsetup{width=1.4\textwidth}
 {\includegraphics[width=\linewidth]{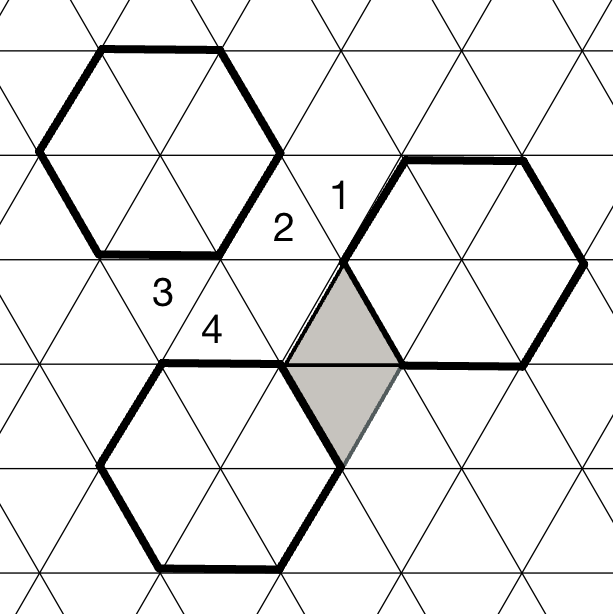}}
    \subcaption{Edge-connected, III}\label{fig:EC3}
  \end{subfigure}\hspace*{.2in}
  \begin{subfigure}[t]{0.17\textwidth}
    \includegraphics[width=\linewidth]{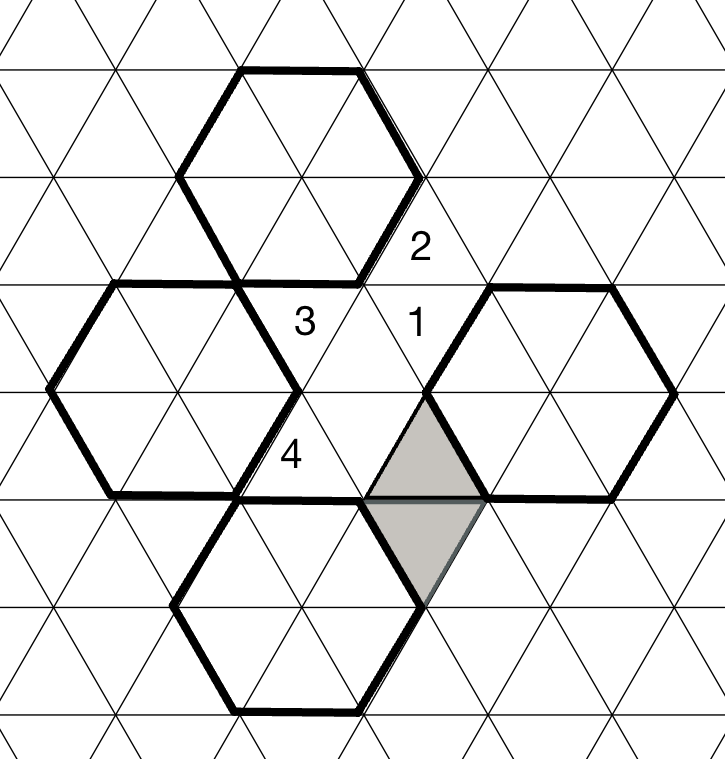}
\captionsetup{width=1.2\textwidth}    \subcaption{Edge-connected, IV }\label{fig:EC4}
  \end{subfigure}
   }
\vspace*{-.1in} \caption{Faces in $\gamma$ at distance 1 in $G_\gamma$ lie on a circuit.}\label{fig:BaseCase}
\end{figure}

Let us prove that $G_\gamma$ is 2-connected. We need to show that no $v\in V(G_\gamma)$ forms a cut or that (by Menger's theorem) for any pair of distinct vertices $u$ and $v$, there is a circuit (i.e., a simple cycle) that contains both of them. We prove this by induction on $d_{G_{\gamma}}(u,v)$, the length of the shortest path between $u$ and $v$ in $G_{\gamma}$.

Assume that $d_{G_{\gamma}}(u,v) = 1$, which occurs when the faces $u$ and $v$ in $\gamma$ are either vertex- or edge-connected. If $u$ and $v$ are vertex-connected (\cref{fig:VC}), then there is a colored hexagon that is edge-connected to both $u$ and $v$. This hexagon belongs to a connected cluster of hexagons of the same color, and the faces edge-connected to its boundary form a circuit that contains $u$ and $v$. Next, consider the situation when $u$ and $v$ are edge-connected. They can be edge-connected to hexagons of the same color, which are necessarily themselves edge-connected,  i.e., they belong to the same monochromatic cluster (\cref{fig:EC1}). In this case, the above argument again applies. It remains to consider the situation when $u$ and $v$ are edge-connected, but the colored hexagons edge-connected to $u$ and $v$, respectively, are of different colors. The possibilities in this case are shown in \cref{fig:EC2,fig:EC3,fig:EC4}.
If one of the two faces (in \cref{fig:EC2}, the top one) is edge-connected to a non-deep hole, then the path that goes through faces 1 and 2 can be completed to a circuit. If face 1 in \cref{fig:EC2} is a deep hole, it can be either
isolated or edge-connected to another deep hole as shown in \cref{fig:EC3,fig:EC4}, respectively (recall also \cref{fig:deep_holes}).  Following the labeled faces in either case enables one to complete a circuit. This proves the base case.

\begin{figure}[h]
\centering
\includegraphics[width=.3\linewidth]{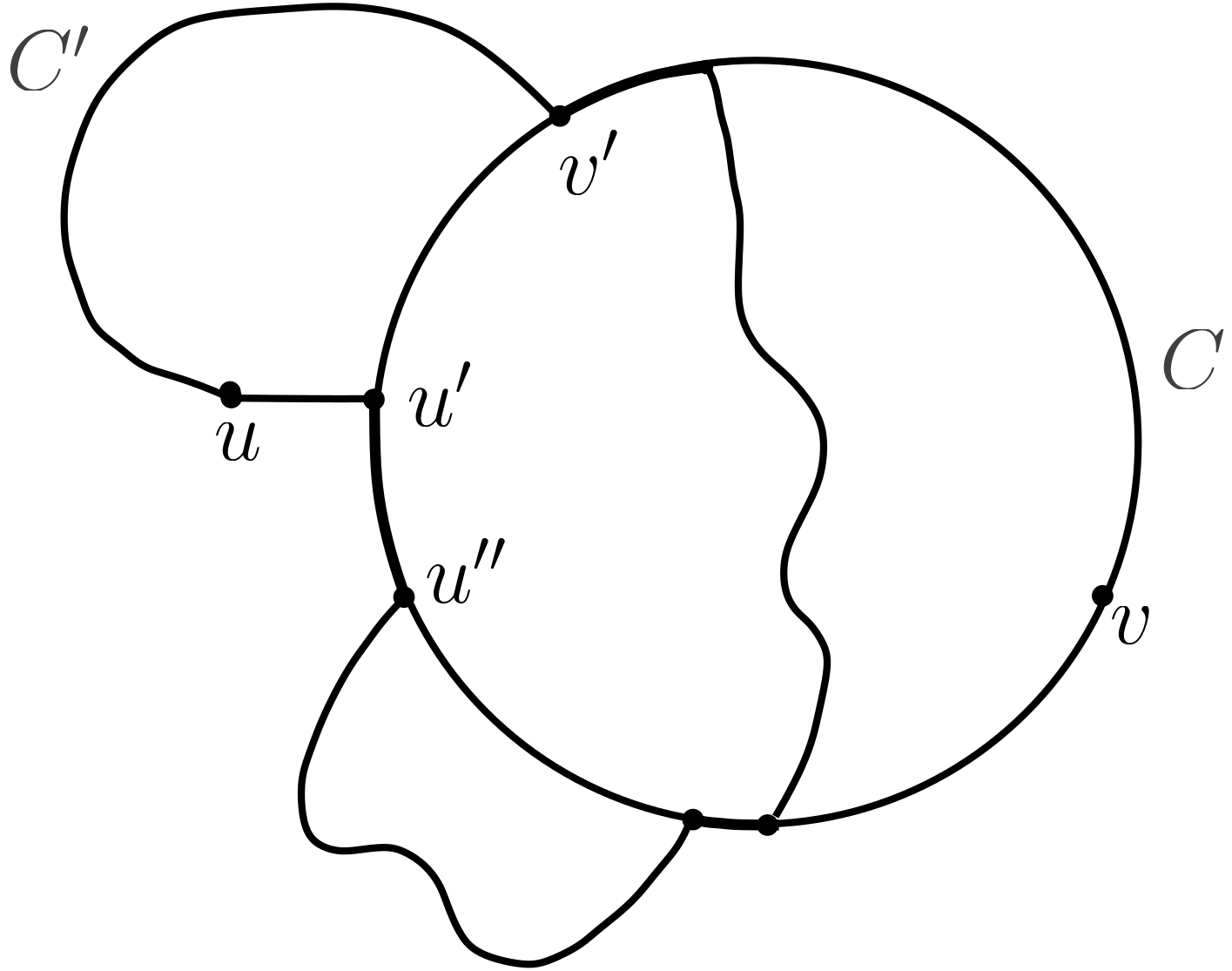}
\caption{The induction step.}\label{fig:induction}
\end{figure}

Now, suppose that, for some $n \ge 1$, each vertex pair $u, v$ with $d_{G_{\gamma}}(u,v) = n$ belongs to a common circuit. Consider $u,v$ with $d_{G_{\gamma}}(u,v) = n+1$, and consider a path from $u$ to $v$ of length $n+1$ in $G_{\gamma}$. If $u'$ is the vertex adjacent to $u$ along this path, then $d_{G_{\gamma}}(u',v) = n$, and an application of the induction hypothesis yields a circuit $C$ containing $u'$ and $v$. Assume that $u\not\in C$. From the base case, there is a circuit $C'$ that goes through the edge $uu'$. 
Note that $C$ and $C'$ have $u'$ as a common vertex and therefore also share an edge on $C$ incident to $u'$ 
(otherwise $u'$ would be of degree 4 in $G_\gamma$, which is impossible). To construct a circuit that contains $u$ and $v$, start from $u'$ and follow the edge $u'u$ along $C'$ until the first vertex common to $C'$ and $C$ (vertex $v'$ in Fig.~\ref{fig:induction}, which may be the same as $u''$). Then follow $C$ via $v$ back to $u'$, completing the circuit and finishing the proof.
\end{proof}

While there are results on the growth rate of the number of 2-connected planar graphs \cite{bender2002number}, the estimates for the count of the graphs they give do not improve on our estimate. The same is true for other known results on the count of planar graphs of various kinds \cite{GimenezNoy2009,NoyEtAl2022,BodirskyEtAl2016}. 

\begin{remark}\label{rem: HA}
    Pirogov--Sinai theory used in the next section for the low-activity regime applies equally to the high-activity case. Moreover, for $\lambda\to\infty$, the ground states of the model are the same as for the
    standard hard-core model. 
    As implied by Theorem~\ref{thm: density}, there are 3 distinct ground states corresponding to the 3  sublattices of $\A$, and one can show that these give rise to 3 distinct extremal periodic Gibbs measures. At the same time, unlike the arguments used above, this approach 
    does not yield specific estimates of $\lambda$ above which there is phase coexistence, and this is why
    we followed the contour counting argument over Pirogov--Sinai theory in the high-activity case.
\end{remark}

\vspace*{-.1in}
\begin{figure}[ht]
  \centering
  \includegraphics[width=\textwidth]{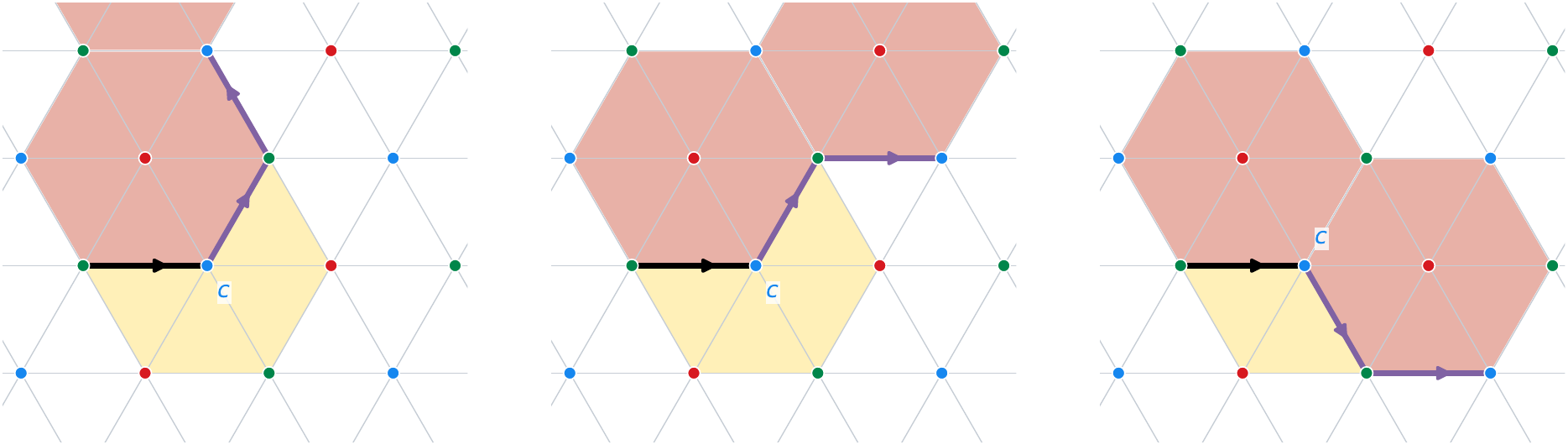}
 \captionsetup{width=.95\textwidth} \caption{The red-shaded hexagons form part of the boundary of the red island, and the contour faces are colored yellow. The (empty) faces of the contour can be counted by following the edges that
  separate the island from its exterior. The figures show a part of the contour formed by 3 edges, and the number of faces can be bounded above by adding the contributions tracked by every other edge. 
  Starting from one edge (that leads to $c$ in the figures), the next step contributes $4$ or $2$ empty faces to the contour depending on which of the $2$ possible directions the path takes, and therefore contributes a weight of $\lambda^{-4/6}$ or $\lambda^{-2/6}$, respectively. The subsequent step has $2$ possible directions, but both add $0$ empty faces, and therefore each contributes $1$ to the weight. Repeating this process yields~\cref{eq:single_island_bound}.}
  \label{fig:triangular-turn-pairs}
\end{figure}

\vspace*{-.25in}
\begin{remark}\label{rem: contour counting} As mentioned in the Introduction, the estimate of $\lambda$ that bounds from below the phase coexistence region obtained in this section is rather weak, especially compared with the results for the maximal hard-core model on the planar integer lattice. Indeed, as shown in \cite{wang2025maximal}, in the case of $\Z^2$, this model gives rise to multiple Gibbs distributions for all $\lambda>6.3506$. This bound relies on a result of \cite{blanca2019phase} which estimates the ``taxi-walk connective constant'', i.e., contours relevant for
proving phase coexistence in the standard hard-core model on $\Z^2$. 
It is likely that a similar estimate can be obtained for $\A$, which would require a more accurate contour counting argument, which we were unable to find.
This difficulty arises from the entropy associated with the possible arrangements of $3$ types of islands.
Consider a contour surrounding a single island of a fixed
non-exterior color whose boundary is a simple self-avoiding polygon in the hexagonal lattice. 
A non-backtracking self-avoiding walk (see~\cref{fig:triangular-turn-pairs}) estimate bounds the contour weight
by 
\begin{equation}\label{eq:single_island_bound}
  \sum_{\substack{\gamma \in \Gamma: \Lambda_m^W \subset \gamma_\bullet \\ \gamma_\bullet \text{contains a single island}}} \lambda^{-\frac{|\gamma|}{6}}\le \sum_{l\ge 0} [2(\lambda^{-2/3}+\lambda^{-1/3})]^l.
\end{equation}
This series converges for $\lambda>10+6\sqrt3\approx20.3923$. 

Extending this boundary-walk estimate to general contours on $\A$ requires
accounting for contacts between islands.
The counting in the proof of Lemma~\ref{lemma: contour counting} discards some of the compatibility constraints among contours and therefore counts impossible pattern assignments.
In particular, we have not used the fact that the size of contours, including deep holes, is a multiple of 6. 
We conjecture that the factor $10^l$ in the estimate of Lemma~\ref{lemma: contour counting} can be replaced with $10^{l/6}$, which would yield a more reasonable lower bound for $\lambda$.
\end{remark}

\section{Phase coexistence: The low-activity regime}
Here we limit our considerations to the low-activity case of the maximal hard-core model on $\A$, so all the statements below should be read with this limitation in mind  (see also Remark~\ref{rem: HA} above). At low activity, Gibbs distributions are determined by the ground states of the model. One naturally expects that for sufficiently small $\lambda$, typical configurations are local perturbations of these ground states. A useful tool for formalizing this idea is provided by the Pirogov-Sinai theory \cite[Ch.~7]{friedliLattice2018}, \cite{zahradnik1998short}.
This theory asserts precisely that when $\lambda$ is sufficiently small, all extremal Gibbs measures arise as thermodynamic limits of finite-volume Gibbs measures with boundary conditions given by (stable) periodic ground states (PGSs). Moreover, configurations sampled from these extremal measures typically differ from the corresponding ground state only by local perturbations.

A configuration $\omega \in \cX(\A)$ is called periodic if there exist two linearly independent vectors $u,v \in \A$ such that $\omega_i = \omega_{i +  v} = \omega_{i+u}$ for all $i \in \A$. 
A PGS is said to generate a Gibbs measure $\mu_\eta$ if 
  $$
        \mu_\eta=\lim_{\Lambda\Uparrow \A}\mu_\Lambda(\cdot|\eta),
        $$
where the limit is in the van Hove sense, assuming that it exists. A Gibbs measure is called {\em extremal} if any decomposition of the
form $\mu=\alpha \nu_1+(1-\alpha)\nu_2$, where $\nu_1,\nu_2$ are Gibbs measures defined by the same specification and $\alpha \in (0,1)$, implies that $\mu=\nu_1=\nu_2$.

The main result in this part is stated next.
\begin{theorem}\label{thm:ps-low} Let $\lambda<1$. There are 14 different PGSs for the maximal hard-core
model on $\A$. For sufficiently small $\lambda$, each ground state $\eta$ gives rise to an extremal periodic Gibbs measure, $\mu_\eta$, and all such measures are generated by the PGSs.  The measures $\mu_\eta$ for different PGSs are mutually singular and have an exponential decay of correlations.
For $\lambda \to 0$ and any $\eta$, the measure $\mu_\eta$ converges weakly to a point mass on $\eta$.
\end{theorem}

The proof proceeds in 3 stages. First we identify the PGSs, then show a Peierls-type bound for the contour size. Once these results are 
established, the remainder of the proof follows by an application of the results of the Pirogov-Sinai theory similar to \cite[Ch.7]{friedliLattice2018} or \cite{mazel2025high}.

Let $\cX^{\text{\rm per}}$ be the set of periodic configurations in $\cX(\A)$. 
For $\omega \in \cX^{\text{\rm per}}$, define the density of $\omega$ as follows:
\[
\delta(\omega) := \lim_{n \to \infty} \frac{|\omega_{\Lambda_n}|}{|\Lambda_n|}.
\]
Here, $|\omega_{\Lambda_n}|$ denotes the number of occupied sites in $\Lambda_n\Subset\A$, as in Def.~\ref{def: MHCM} (the limit clearly exists). The following quantity is called the \emph{energy density}:
\[
e(\omega) := \lim_{n \to \infty} \frac{1}{|\Lambda_n|} \cH_{\Lambda_n}(\omega) = - (\ln \lambda) \delta(\omega),
\]
where $\cH_{\Lambda}(\omega) = - |\omega_\Lambda| \ln \lambda $ is the Hamiltonian.

It is known~\cite[Lemma 7.4]{friedliLattice2018} that $\eta \in \cX(\A)^{\text{\rm per}}$ is a PGS if and only if its energy density is minimal:
\[
e(\eta) = \inf_{\omega \in \cX(\A)^{\text{\rm per}}} e(\omega).
\]

Depending on whether $\lambda>1$ or not, $e(\omega)$ is minimized by maximizing or minimizing $\delta(\omega)$. In other
words, PGSs are equivalently defined as the sparsest (low activity) and densest (high activity) configurations in $\cX(\A)^{\text{\rm per}}$.

\begin{figure}[h]
\centering \includegraphics[width=.18\linewidth]{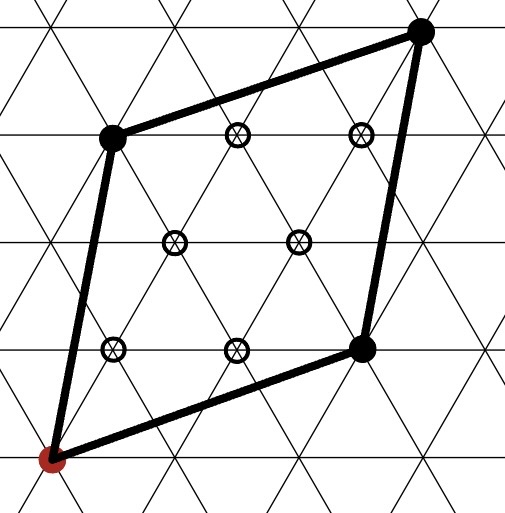}
  \caption{Sparse PGSs for $\cX(\A)$ of density 1/7. }\label{fig:PGS}
\end{figure}

\begin{theorem}\label{thm: density}
 The density of  periodic configurations of $\cX(\A)$ satisfies $\frac{1}{7} \le \delta(\omega) \le \frac{1}{3}$. 
    There are exactly $3$ periodic configurations with density $\frac{1}{3}$ and $14$ periodic configurations with density $\frac{1}{7}$.
\end{theorem}

\begin{proof}
Bounds on the density are proved by constructing tilings of the plane with Voronoi regions of the occupied vertices and 
counting the proportion of covered vertices. As before, we represent vertices in $\A$ using their coordinates in the basis $(\bfb_1,\bfb_2)$.
Let $\omega \in \cX(\A)^{\text{\rm per}}$ be a periodic configuration such that $\omega_{x + v_1 \bfb_1} = \omega_{x + v_2 \bfb_2} = \omega_x$ for every $x \in \A$ for some positive integers $v_1, v_2$,

and then 
\[
\delta(\omega) = \frac{|\omega_{{v_1 \times v_2}}|}{v_1 v_2},
\]
where (by abuse of notation) $\omega_{v_1 \times v_2}$ denotes $\omega$ restricted to the parallelogram. 
The first step is to use periodicity to identify this tile with the torus $\torus_{v_1,v_2}=\{(i,j) \mid i \in \Z\, (\text {mod\,} v_1), j \in \Z\, (\text {mod\,} v_2)\}$, so that every $0$ in $\omega_{v_1 \times v_2}$ is adjacent to at least one $1$. Next, let us count
the number, $N$, of pairs $(u,v)$ in which $u$ and $v$ are adjacent sites in $\torus_{v_1,v_2}$, 
with the property that $\omega_u = 1$ and $\omega_v = 0$. Since every occupied site $u \in \supp(\omega_{v_1 \times v_2})$ is adjacent to six unoccupied sites, we have $N = 6  |\omega_{v_1 \times v_2}|$. 
On the other hand, every unoccupied site $v$ in the torus $\torus_{v_1,v_2}$ is adjacent to at least one and at most three occupied sites; so we also have $v_1v_2 - |\omega_{v_1 \times v_2}| \le N \le 3 (v_1v_2 - |\omega_{v_1 \times v_2}|)$. 
Substituting for $N$, we obtain
     $$
v_1v_2 - |\omega_{v_1 \times v_2}| \le 6  |\omega_{v_1 \times v_2}| \le 3  (v_1v_2 - |\omega_{v_1 \times v_2}|),
     $$
which implies that $\frac17 \le \delta(\omega) \le \frac13$. Moreover, we have $\delta(\omega) = \frac17$ iff equality holds in the left inequality above, i.e., every unoccupied site $v$ in the torus $\torus_{v_1,v_2}$ is adjacent to exactly one occupied site. 
Similarly, we have $\delta(\omega) = \frac13$ iff equality holds in the right inequality, i.e., every unoccupied site $v$ in the torus $\torus_{v_1,v_2}$ is adjacent to exactly three occupied sites. 

The maximal-density PGSs for $\A$ are the three configurations with all the same-color vertices occupied. 
To identify the lowest-density PGSs, we note that periodicity implies the existence of a tiling of the (discrete) plane. Indeed, $\A$ can be tiled with parallelograms, as shown in~\cref{fig:PGS-A}, which also shows a
PGS for $\A$. Six other PGSs are obtained as shifts of the one containing the origin, shown as the red-colored 
site in \cref{fig:PGS} (the 6 unfilled sites correspond to the shifts). Together with the mirror images, the maximal hard-core model has 14 sparse PGSs. 
\end{proof}

 \vspace*{.15in}
\begin{figure}[h]
\centering \includegraphics[width=.4\linewidth]{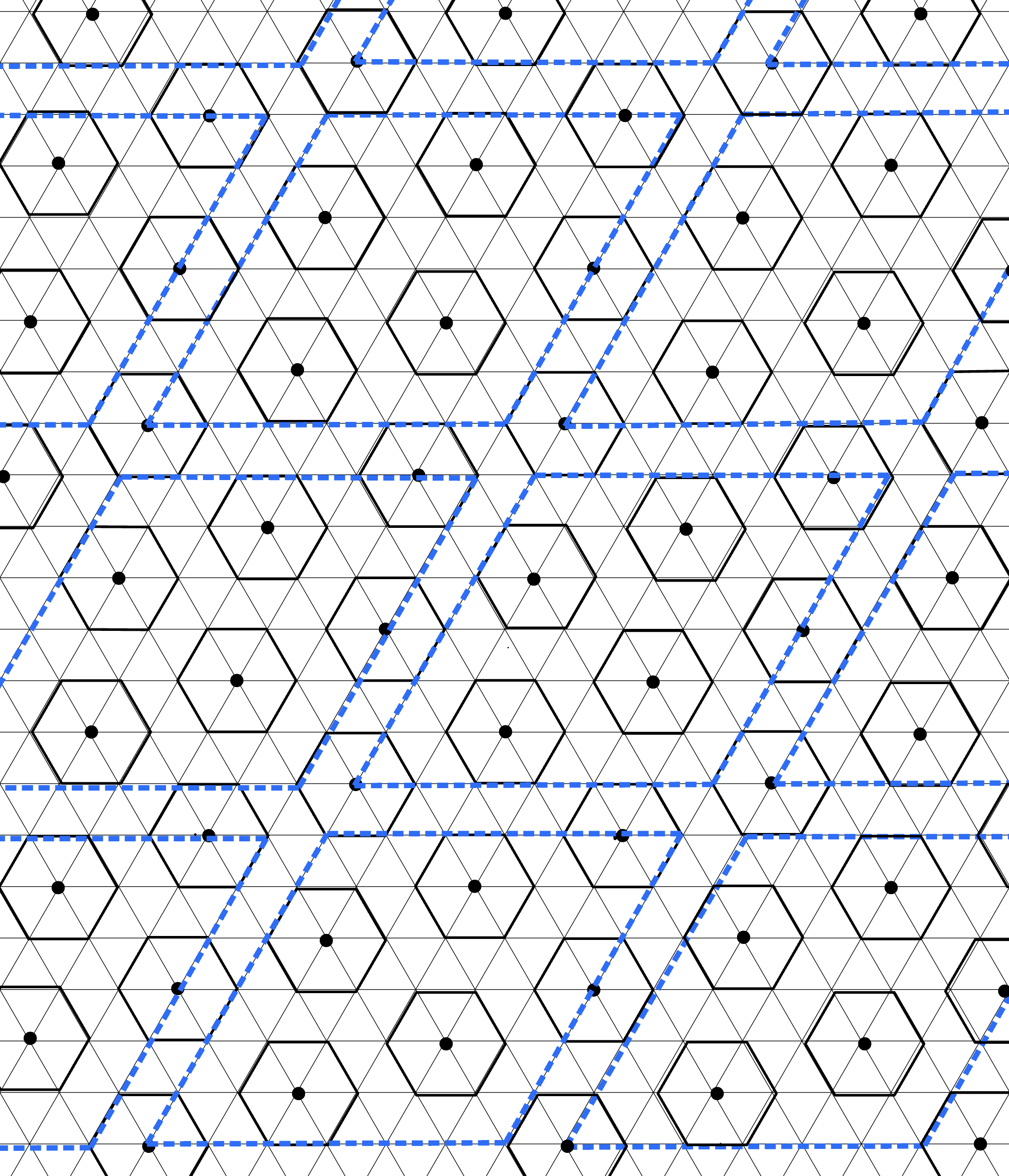}
  \caption{A sparse PGS: each unoccupied vertex is adjacent to {\em exactly one} occupied vertex. The figure shows one of the 14 PGSs for $\A$ of density 1/7. The $7\times 7$ parallelograms form a tiling of $\A$.} 
\label{fig:PGS-A}
\end{figure}

The main part of the proof of \cref{thm:ps-low} is devoted to verifying a Peierls-type condition \cite[Sec.1.5]{zahradnik1998short}, \cite[p.357]{friedliLattice2018}. Loosely speaking, this condition controls the increase of energy caused by a perturbation of a ground state, and is satisfied if this increase is proportional to the size of the perturbation. This condition can be rigorously expressed using the contour formalism.

Let $\cQ$ be the set of all $7 \times 7$ patterns that occur in $\cX(\A)$ and let us partition $\A$ into disjoint $7 \times 7$ blocks. Namely, let
$$
   \cP_0:=\{(x,y)\in \A : 0 \le x <7, 0 \le y < 7\}
   $$
be the block with the lower-left corner at the origin (recall that we are using the basis $\bfb_1,\bfb_2$).
We may then reinterpret configurations in $\cX(\A)$ using a block-spin representation: consider a partition of $\A$ into 
blocks of the form  $\cQ(l,m) := \cP_0+ 7(l,m)$, where $(l,m) \in \A$, and encode each block with an element of $\cQ$. See \cref{fig:PGS-A}; note that $\cQ(l,m)$ is the $7 \times 7$ box with the bottom-left corner at $(7l,7m)$. 

Let $\mathcal{G}$ denote the set of 14 periodic ground states for $\lambda < 1$ and let $\cQ_{\text p}\subset \cQ$ be the set of $7 \times 7$ patterns arising from the periodic ground states. 
Then, in this block-spin representation, each ground state corresponds to a constant configuration $Q^{\Z^2}$ for some $Q \in \cQ_{\text p}$.

Let $\omega \in \cX(\A)$ and $\bar{\omega}$ be the corresponding block-spin representation.
Denote by $\bar{\omega}_{l,m}$ the restriction of $\omega$ to $\cQ(l,m)$. 

\begin{definition}
Let $Q \in \cQ_{\text p}$.
The block $\cQ(l,m)$ is called \emph{$Q$-correct} in $\omega$ if $\bar{\omega}_{l + i,m + j} = Q$ for all $i, j \in \{0,\pm 1\}$. 
In other words, a correct block at $(l,m)$  matches the pattern $Q$ together with all of its 8 neighboring blocks.
We call a block {\em incorrect} if it is not $Q$-correct for any $Q \in \cQ_{\text p}$. 
\end{definition}

We define a {\em contour} in $\omega$ as a maximal connected component $\gamma$ formed of incorrect blocks\footnote{
Here we slightly abuse terminology: this notion of a contour is different from the one introduced in Definition~\ref{def:contour}.}. 
The energy associated with a contour $\gamma$ is determined by the number of occupied vertices that $\gamma$ encloses. 
The Hamiltonian restricted to $\gamma$ is
   $$
    \cH_{\gamma}(\omega) = - |\omega_\gamma| \ln \lambda.
   $$
Let $\eta$ be one of the periodic ground states. The relative Hamiltonian of $\omega$ with respect to $\eta$, restricted to $\gamma$, is given by 
   $
    \cH_\gamma(\omega \mid \eta) =  (|\eta_\gamma| - |\omega_\gamma|) \ln \lambda.
   $
Below, we will assume that $\omega$ is a local perturbation of $\eta$, i.e., they differ in a finite number of sites.
By definition of the ground state, $\cH_\gamma(\omega \mid \eta)\ge 0$. The Peierls condition, moreover, requires that
   $$
\cH_\gamma(\omega \mid \eta) \ge \rho |\gamma|
   $$
for some constant $\rho > 0$. Using the expression for $\cH$ and recalling that $\lambda<1$, we can write this explicitly as
\begin{equation}\label{eq:peierls}
    |\omega_\gamma| - |\eta_\gamma| \ge \alpha |\gamma|,
\end{equation}
where $\alpha := - \frac{\rho}{\ln \lambda} > 0$.

\begin{remark}
    There is a technical issue concerning the way of verifying the Peierls condition. Often $\gamma$ is defined as the ``thick'' boundary of all the incorrect blocks within $\cQ(l,m)$, i.e., the union of all contours in the
    region, see e.g., \cite[Ch.7]{friedliLattice2018}. Another possibility appearing in the literature  (\cite[Sec.1.5]{zahradnik1998short}, \cite{mazel2019high}) is verifying condition \eqref{eq:peierls} for individual contours, and then aggregating these inequalities for the final result. We will follow the second option without further mention.
\end{remark}

One approach to analyzing the number of occupied vertices in a configuration $\omega\in\cX(\A)$ is through the \emph{Delaunay triangulation} defined by its support. 
Let $\omega \in \cX(\A)$.
To every occupied site, $i$,  there corresponds a polygon formed of the points of $\R^2$ that are closer to $i$ than to any other occupied site of $\omega$. 
The collection of these polygons forms a Voronoi partition of $\R^2$ defined by $\omega$. 
The Delaunay triangulation defined by $\omega$ is a triangulation of $\R^2$ constructed as the dual graph of the Voronoi partition defined by $\omega$; namely, the circumcenters of the triangles are the vertices of the Voronoi diagram. The vertices of the triangles are the sites $i\in \A$ such that $\omega_i=1$. Importantly, Delaunay triangulations have the following characterization: the circumcircles of the triangles do not contain any occupied sites of $\omega$ other than the vertices of their triangles.

Call a triangle with vertices at the lattice points {\em regular} if it is equilateral of side length $\sqrt{7}$ and call it {\em defective} otherwise.
 If $\omega$ is one of the periodic ground states, then every Delaunay triangle is regular. Call a triangle with vertices $u,v,w\in\A$ {\em feasible} if (1) its circumcircle does not contain a
hexagon with side length 1 centered at a vertex of $\A$, (2) none of the pairs among $u,v,w$ are neighbors in $\A$. Delaunay triangles are clearly feasible.

The following statement is proved by examining all possible triangles in a Delaunay triangulation of a MIS. 
\begin{lemma}\label{prop: 5/2}
   For $\omega\in\cX(\A)$, the area of any defective triangle in a Delaunay triangulation defined by $\omega$ is at most $\frac{3\sqrt{3}}{2}$.
\end{lemma}
\begin{proof}
Below, we do not distinguish $\omega$ from its support.
We first claim that any closed circle in $\R^2$ of radius $\rho=\sqrt{7/3}\approx 1.527$ contains a point in $\omega$.
Let $C$ be such a circle. Its center may not lie in $\A$ but is at distance $\le \sqrt 3/3$ from one of the lattice points. This point is either in $\omega$
or at distance 1 from a point in $\omega$, so there is an occupied vertex, $v$, no further than $\rho$ from the center of the circle.

We will examine all the feasible triangles with circumradius $\le \rho$ with vertices in $\omega$. 
It suffices to examine triangles with one vertex at $(0,0)$ since if there is a MIS with a triangle of some shape elsewhere, its shift is also a MIS. Therefore, all the triangles will be assumed to have one of their vertices at the origin. 
The remaining two vertices are contained in the circle of radius $2\rho$ with center at the origin, including possibly on its circumference. There are $36$ such vertices in $\A$ not counting the origin. Examining all the possible triangles (with the help of computer) proves the lemma.
\end{proof}

Let $\torus_k := \torus_{7k,7k}$ be a $7k \times 7k$ torus in $\A$ as defined in the proof of \cref{thm: density}, and $\omega$ be an admissible configuration on $\torus_k$, meaning that $\omega$ extends to a $7k$-periodic configuration in $\cX(\A)$.
Denote by $\sT_\omega(\torus_k)$ a Delaunay triangulation of $\torus_k$ defined by $\omega$.

\begin{lemma}  For any $\torus_k$-admissible configuration $\omega$,
\begin{enumerate}
    \item[(a)] the triangles in $\sT_\omega(\torus_k)$ form a tessellation of $\torus_k$;
    \item[(b)] the number of triangles $|\sT_\omega(\torus_k)|=2|\omega_{\torus_k}|$.
 \end{enumerate}
 \end{lemma}
\begin{proof}
    The first part follows by definition. The second claim is justified as follows. Place a disk of area $\pi/8$ around each occupied site; since the distance between the sites is at least $1$, the disks do not overlap. Each triangle intersects exactly three such disks---one for each of its vertices. Moreover, since the sum of the interior angles is $\pi,$ the sum of the areas of the three sectors cut out by the triangle is  $\pi/16$.  Rephrasing, the area within the disks covered by each triangle is exactly one half of a single disk. 
     On the other hand, every disk is fully covered by the triangles, so the count of triangles is twice the number of disks, which is also the number of occupied sites.
\end{proof}

\begin{definition}
An incorrect block is said to be t-{\em defective} if it intersects at least one defective triangle (even on a single vertex). 
An incorrect block is called n-{\em defective} if it is not t-defective: these blocks are neighbors of a t-defective block, which themselves take values from $\cQ_{\text p}$. The neighboring blocks of $\cQ(l,m)$ are $\cQ(l+i,m+j)$ with $(i,j)\in\{-1,0,1\}^2\setminus\{(0,0)\}$. 
\end{definition}

The specific version of \eqref{eq:peierls} that we establish appears next.
\begin{theorem}\label{thm: Peierls} Let $\omega\in \cX(\A)$, let $\eta$ be a periodic ground state, and let $\gamma$ be a contour formed of incorrect blocks. Then
  $$
  |\omega_\gamma|-|\eta_\gamma|\ge \max\Big(1, \frac1{378}|\gamma|\Big).
  $$
\end{theorem}
\begin{proof}
Observe that the blocks in the immediate neighborhood of $\gamma$ are $Q$-correct for some values of $Q \in \cQ_{\text p}$.
Recall that the $Q$-constant block-spin configuration is a periodic ground state. Therefore, $\omega_\gamma$ can be extended using $Q$-tiles, which gives us a configuration  $\tau$ on $\torus_k$ for a  sufficiently large $k$. 
In other words, we define a configuration $\tau$ on $\torus_k$ such that
$\tau_\gamma=\omega_\gamma$ and $\tau_{\gamma^c}=\eta'_{\gamma^c}$ for some other periodic ground 
state of our model
\footnote{If $\gamma$ is not
simply connected, it may happen that $\omega_{\gamma^c}$ agrees with different periodic ground states in different connected
components of $\torus_k$ defined by $\gamma$. This will not matter for the argument below, since we rely on the number of occupied sites rather than
on the ground states themselves.}. Importantly, $|\tau_{\gamma^c}|=|\eta_{\gamma^c}|$, where $\gamma^c=\torus_k\backslash\gamma$.
    
By Lemma~\ref{prop: 5/2}, any defective triangle has area at most $\frac{3\sqrt{3}}{2}$. If we replace one or more regular triangles (which have area $\frac{7\sqrt{3}}{4}$) with defective (smaller) ones, then to tile $\torus_k$ we will use more triangles than the number of triangles in 
 $\sT_\eta(\torus_k)$ (this number equals 
$\frac{|\torus_k|}{7 \sqrt{3} / 4}= 14k^2$, where $|\cdot|$ denotes the area).  
The total number of triangles in a triangulation is even (it is twice the number of occupied sites), so the total count of triangles increases by at least $2$, and the number of occupied sites increases by at least $1$. 

The size of $\sT_\tau(\torus_k)$ can be estimated as follows. Denote by $R$ a generic triangle and let 
$N_0(\tau)$ and $N_d(\tau)$ be the number of regular and defective triangles in $\sT_\tau(\torus_k)$, respectively. We have
   \begin{align*}
   |\torus_k|&=\sum_{\text{defective } R} |R| + \frac{7 \sqrt{3}}{4} N_0(\tau) \\
   & \le \, \frac{3\sqrt{3}}{2} N_d(\tau) + \frac{7 \sqrt{3}}{4} N_0(\tau)
   \end{align*}
using the fact that the area of a defective triangle is at most $\frac{3\sqrt{3}}{2}$, proved in Lemma~\ref{prop: 5/2}. Rearranging this inequality, we obtain 
   $$
    N_0(\tau) + N_d(\tau) \ge \frac{|\torus_k|}{7 \sqrt{3}/4} + \frac17 N_d(\tau).
   $$
Thus, the triangulation $\sT_\tau$ increases the count of triangles over $\sT_\eta$, namely, 
    $$
    |\sT_\tau(\torus_k)|\ge |\sT_\eta(\torus_k)|+ \frac17 N_d(\tau).
    $$
Since each $t$-defective incorrect block is surrounded by at most $8$ $n$-defective incorrect blocks (\cref{fig:PGS-A}), all of which enter
$\gamma,$ and each defective triangle
intersects at most 3 incorrect blocks, we obtain
   $$
    3 N_d(\tau) \ge \sharp\{\text{t-defective blocks}\} \ge \frac{|\gamma|}{9}.
   $$
Note that $\omega_\gamma=\tau_\gamma$ (since the only changes are outside $\gamma$) and $|\omega_{\gamma^c}|\ge |\tau_{\gamma^c}|$ (since $\omega_{\gamma^c}$ may span defective triangles, and $\tau_{\gamma^c}$ gives rise only to regular ones). 
We continue as follows:
  \begin{align*}
      |\omega_{\torus_k}|&=|\omega_\gamma|+|\omega_{\gamma^c}|\ge |\tau_\gamma|+|\tau_{\gamma^c}|\\
      &=\frac12|\sT_\tau(\torus_k)|\ge\frac12\left(|\sT_\eta(\torus_k)|+\frac{N_d(\tau)}7\right)\\
      &\ge \frac12\Big(|\sT_\eta(\torus_k)|+\frac{|\gamma|}{189}\Big)\\
      &=|\eta_{\torus_k}|+\frac{|\gamma|}{378}. 
  \end{align*}
Since we argued earlier that $|\omega_{\torus_k}|\ge |\eta_{\torus_k}|+1$, the proof is complete.
\end{proof}

The first claim of Theorem~\ref{thm:ps-low} was proved in Theorem~\ref{thm: density} above; thus, in the maximal hard-core case at low activity, periodic ground states form a single equivalence class defined by lattice symmetries (shifts and reflections). 
The remaining parts of the proof follow by an application of the Pirogov-Sinai theory. The claim about the existence of Gibbs measures $\mu_\eta$ follows from Theorem 7.41 in \cite{friedliLattice2018}. There is a subtle point here, namely that
a priori, not all PGSs give rise to extremal Gibbs measures, although one of them necessarily does. Moreover, once such a state is found, the states that fall in the same equivalence class as determined by lattice symmetries also generate extremal Gibbs measures. To argue that these measures are distinct, we note that for a given PGS $\eta$ and $(l,m) \in \Z^2$, any configuration $\omega$ with $\omega_{\cQ(l,m)} \ne \eta_{\cQ(l,m)}$ must contain a contour that either intersects or surrounds $\cQ(l,m)$, and the probability of such events admits a summable bound. Then the Peierls
argument implies 
  \begin{equation}\label{eq:crystallization}
\mu_\eta(\omega_{\cQ(l,m)} = \eta_{\cQ(l,m)}) \ge 1 - \epsilon(\lambda),
\end{equation}
where $\epsilon(\lambda) > 0$ tends to zero as $\lambda \to 0$; see, e.g., \cite[Thm.~7.41]{friedliLattice2018}. The same theorem also proves extremality of the obtained Gibbs measures relying on a characterization via the absence of long-range correlations, which also leads to exponential decay of correlations~\cite[Cor. 7.43]{friedliLattice2018}.
That these measures are periodic follows from \cite[Theorem 1]{dobrushin1985problem}. 
 Note that by~\eqref{eq:crystallization}, the measure $\mu_\eta$ converges to  the point mass at $\eta$ as $\lambda \to 0$. 
Moreover, the structure of the Pirogov-Sinai argument implies that distinct extremal Gibbs measures are pairwise singular.

An accessible presentation of the entire argument, including a complete proof, can be found, for instance, in Chapter 7 of \cite{friedliLattice2018}. A similar argument, without proof, but with detailed references to Zahradnik's version of the Pirogov-Sinai theory \cite{zahradnik1984alternate}, is given in \cite[Theorem III]{mazel2025high}.

\begin{remark}
   A natural counterpart of the triangular lattice $\A$ is given by its dual, the honeycomb lattice $\mathbb{H}$ (technically, $\mathbb H$ is a union of two triangular lattices and does not itself form a lattice). The maximal hard-core model is easily defined, and one can address the same set of questions as those studied for 
   $\Z^2$ and $\A$. The low-activity case can be analyzed with little difficulty, relying on the same set of ideas. In particular, the least-populated PGSs of the maximal hard-core model on $\mathbb{H}$ have density $1/4$ since every occupied vertex has exactly 3 unoccupied neighbors, and every unoccupied vertex is connected to exactly one occupied  vertex.
At the same time, the high-activity case is less obvious because the contour-erasing approach used here
encounters difficulties (construction of the mapping $\phi$ used to erase the contours is not immediate). It is possible that cluster expansions for the hard-core model, such as those studied in \cite{jauslin2018high,he2024high},
could shed more light on the problem of phase coexistence for the maximal model on $\mathbb H$.
\end{remark}

\section{Capacity of the system}\label{sec:capacity}
In this section, we develop another perspective of the maximal hard-core model related to information encoding into MISs on $\A$, which we
again view as configurations in $\cX(\A)$. A {\em recovery condition} is a function that determines the value of a site, $i$, from the values of its neighbors, and a collection of all configurations that follow a particular recovery condition is called a recoverable system. Motivated by the concept of codes with local reconstruction \cite{Ramkumar2022}, one-dimensional recoverable systems were introduced in \cite{elishcoRecoverableSystems2022}, both in the deterministic and probabilistic settings. 
Given a configuration $\omega\in\{0,1\}^\A$, we consider the following recovery condition: for any site $i\in\A$, the value $\omega_i=1$ if and only if $\omega_j=0$ for all $j\sim i$, i.e., all the sites $j$ adjacent to $i$ in $G_\A$. 
It is easy to see that this is an equivalent description of configurations $\omega\in \cX(\A)$, and the support of any such $\omega$ forms a MIS in $\A$. 

The capacity or topological entropy of the system $\cX(\A)$ is defined through a limiting procedure.
A configuration $\eta$ in a finite region $\Lambda \Subset \A$ is called {\em compatible} if there exists $\omega \in \cX(\A)$ such that $\eta = \omega_{\Lambda}$, and it is called {\em maximal} if every $0$ is adjacent to at least one $1$ and it does not contain adjacent $1$'s. 
Equivalently, the support of {a maximal configuration} $\eta$ is a MIS in the subgraph induced by $\Lambda$. 
The topological entropy of $\cX(\A)$ is defined as 
\begin{equation}\label{eq:top_entropy}
    H_0 := \lim_{n \to \infty} \frac{\log_2 |\cX_{\Lambda_n} (\A)|}{|\Lambda_n|},
\end{equation}
where $\cX_{\Lambda_n}(\A)$ is the collection of compatible configurations on $\Lambda_n$, and $|\Lambda_n|$ denotes the number of sites in $\Lambda_n$.
The limit in \cref{eq:top_entropy} exists due to subadditivity. We note that in the language of information theory, $H_0$ is called the {\em capacity} of the system.

\subsection{Relation to subshifts of finite type}
The fact that the system $\cX(\A)$ has positive capacity follows from its description as a subshift of finite type (SFT), i.e., a collection of configurations $\cX \subset \{0,1\}^{\A}$ defined by forbidding finitely many local patterns.

It is easy to see that $\cX(\A)$ is equivalent to an SFT defined by forbidding the following patterns:
\[
\begin{array}{c@{\hspace*{0.05in}}c}
    1 & 1
\end{array} \quad
\begin{array}{c@{\hspace*{0.05in}}c}
     1& \\
    &1 
\end{array}  \quad
\begin{array}{c@{\hspace*{0.05in}}c}
    & 1 \\
    1 & 
\end{array} \quad 
\begin{array}{c*4{@{\hspace*{0.03in}}c}} 
    & 0 & &0 &\\
    0 & &0 && 0 \\
    & 0 & &0 & 
\end{array}.
\]

An SFT is called {\em strongly irreducible} if there exists an $r>0$ such that for all finite regions 
$A,B\Subset \A$ with  graph distance $\ge r$, for every pair of configurations $\omega,\eta\in\cX$, there
exists $\sigma\in \cX$ such that $\sigma_A=\omega_A$ and $\sigma_B=\eta_B$. 
As shown in \cite{burton1994nonuniqueness}, a strongly irreducible SFT on a 2D lattice with at least $2$ elements has positive topological entropy.

\begin{lemma}\label{lemma:induction} A maximal configuration $\eta$ on $U_{n,m}$ with $n,m\ge 1$ is compatible, where $U_{n,m} := \{1,\dots, n\} \times \{1, \dots, m\}$.
\end{lemma}
\begin{proof} Consider the enlarged rectangle, $U_{n+1,m+1}$. The support of $\eta$ forms an independent set
in it, which can be extended to a MIS in $U_{n+1,m+1}$. By induction, we conclude that there exists an $\omega\in \cX(\A)$ such that $\omega_{U_{n,m}}=\eta$.
\end{proof}
\begin{proposition}\label{lemma:strongly_irreducible}
\begin{enumerate}[label=(\alph*),wide=0pt,nosep]
  \item The recoverable system $\cX(\A)$ is a strongly irreducible SFT with $r=4$. 
  \item Let $\sigma$ and $\eta$ be maximal configurations in 
$U_{n,m}$ and $(0,m+1) + U_{n,m}$, respectively. There exists a maximal compatible configuration $\omega$ on $U_{n,2m+1}$ such that $\omega_{U_{n,m}} = \sigma$ and $\omega_{(0,m+1) + U_{n,m}} = \eta$.
\end{enumerate}
\end{proposition}
\begin{proof}
Strong irreducibility can be proved by noticing that the argument in Prop.~A.2 in \cite{wang2025maximal}, written there for the $\Z^2$ lattice, works just as well for the triangular lattice. 

To prove (b), let $\bar\omega$ be a configuration on $U_{n,2m+1}$ such that $\bar\omega_{U_{n,m}} = \sigma$, $\bar\omega_{(0, m+1) + U_{n,m}} = \eta$ and $\bar\omega_i=0$ for all $i\in (U_{n,m}\cup(0, m+1)+U_{n,m})^c$.
The configuration $\bar\omega$ corresponds to an independent set in $U_{n,2m+1}$, which must be contained in a MIS in $U_{n,2m+1}$. 
Let $\omega$ be such that its support forms this MIS. Observe that $\bar\omega_i=\omega_i$ for all $i\in U_{n,2m+1}$ except possibly the sites $i \in\{1,\dots,n\}\times \{m+1\}$.
By \cref{lemma:induction}, $\omega$ is compatible.
\end{proof}

A shift-invariant maximal hard-core Gibbs measure with $\lambda = 1$ maximizes the measure-theoretic entropy, which we proceed to define. Let $P$ be a {\em shift-invariant} probability measure on $\cX(\A)$. 
The {\em entropy rate} of $P$ is defined by 
\begin{equation*}
    H(P) = \lim_{n \to \infty} - \frac{1}{|\Lambda_n|} \sum_{\eta \in \cX_{\Lambda_n}(\A)} P(\Pi^{-1}_{\Lambda_n}(\eta)) \log_2 P(\Pi^{-1}_{\Lambda_n}(\eta))
\end{equation*}
where $\Pi_{\Lambda_n} : \cX(\A) \to \cX_{\Lambda_n}(\A)$ is the projection map $\Pi_{\Lambda_n}(\omega) := \omega_{\Lambda_n}$, and $\Pi^{-1}_{\Lambda_n}(\eta) = \{\omega \in \cX(\A) : \omega_{\Lambda_n} = \eta\}$ is a cylinder set.

The variational principle implies that $H_0 = \sup_{P} H(P)$, where the supremum is taken over all shift-invariant probability measures $P$ on $\cX(\A)$ \cite[Sec.~4.5]{katok1995introduction}.
Moreover, the supremum is achieved if and only if $P$ has {\em uniform conditional probabilities}.
That is, given $\omega \in \cX(\A)$, for every finite region $\Lambda \Subset \A$, $P_{\Lambda}(\cdot \mid \omega)$ is a uniform distribution on $\{\eta \in \cX(\A) : \eta_{\Lambda^c} = \omega_{\Lambda^c}\}$ (see~\cite[Proposition~1.8]{haggstrom1996phase}).

\begin{proposition}
    Let $\mu$ be a shift-invariant maximal hard-core Gibbs measure with $\lambda=1$. Then $H(\mu) = H_0$.
\end{proposition}
\begin{proof}
The proof is immediate from~\eqref{eq:def_hard-core_mis}, since $\mu_{\Lambda}(\cdot \mid \omega)$ is a uniform distribution when $\lambda = 1$. 
\end{proof}
Observe that a shift-invariant maximal hard-core Gibbs measure can be prepared with a periodic boundary condition. 

\subsection{Capacity and encoding schemes}
In~\cite{wang2025maximal} we showed that the capacity of $\cX(\Z^2)$ satisfies $0.3012\le H_0\le 0.3409$ relying on a standard transfer matrix computation for the 2-dimensional SFT; see, e.g.,~\cite{forchhammer2000}. Here we establish analogous bounds for $\cX(\A)$.
\begin{proposition}\label{prop:H_0}
The topological entropy of $\cX(\A)$ satisfies $0.3233 \le H_0 \le 0.3536$. 
\end{proposition}
\begin{proof}[Proof sketch]

Fix $n,m\ge 1$. For each $k\ge 1$, place independently chosen
configurations from $\cM_{U_{n,m}}(\A)$ in the $k^2$ blocks
\[
(a(n+1),b(m+1))+U_{n,m}, \qquad 0\le a,b<k
\]
in the top-left corner of each block.
These blocks are separated by one-site gaps, so their occupied sites
form an independent set, which can be extended to a MIS on $\A$. Every unoccupied site in each block already has an occupied
neighbor within that block. Hence the extension preserves all selected
block configurations, and distinct choices give distinct compatible
patterns on $U_{k(n+1),k(m+1)}$. Therefore,
\[
|\cX_{U_{k(n+1),k(m+1)}}(\A)|
\ge |\cM_{U_{n,m}}(\A)|^{k^2}.
\]
Taking logarithms, dividing by the number of sites, and letting
$k\to\infty$, we obtain
\begin{equation}\label{eq:H_0-low}
H_0 = \lim_{k\to\infty}
\frac{\log_2|\cX_{U_{k(n+1),k(m+1)}}(\A)|}
     {k^2(n+1)(m+1)} \ge
\frac{\log_2|\cM_{U_{n,m}}(\A)|}{(n+1)(m+1)}.
\end{equation}

The number $|\cM_{U_{n,m}}(\A)|$ of MISs on $U_{n,m}$ can be computed by an application of the transfer matrix method as discussed in~\cite[App. D.1]{wang2025maximal}. To obtain the lower bound in the claim, we take $n = 10$ and $m = 132$ and perform the calculation by computer.

The upper bound is obtained by relying on the results of~\cite{forchhammer2000}, a detailed explanation can be found at~\cite{wang2025maximal}.
\end{proof}

\emph{Encoding procedure:} The problem of encoding a constrained system has often been considered in information theory literature based on Markov conditions or tiling considerations \cite{roth2002efficient,sharov2010two}. We confine ourselves to the following simple remark. Let $N := \lfloor \log_2 |\cM_{U_{n,m}}(\A)| \rfloor$. To encode an infinite bit string $x$ into a configuration $\omega\in\cX(\A)$, we begin with establishing an arbitrary bijection $\phi$ between segments $s_t=\{tN, \dots, {(t+1)N-1}\}, t\in \Z$ and pairs of integers $(t_1,t_2)\in \Z^2$. Let $f :\{0,1\}^N \to \cM_{U_{n,m}}(\A)$ be a function that encodes an $N$-bit string
into a MIS on $U_{n,m}$. For each $t \in \Z$, we place 
   \(
   f({x_{s_t}})
   \)
in the top-left corner of the block $(t_1 (n+1), t_2(m+1)) + U_{n+1,m+1}$, where $t$ and $(t_1,t_2)$ 
are related through $\phi$. The unfilled lines between the $n\times m$ blocks can be filled by a recursive application of \cref{lemma:strongly_irreducible} (b). As $n,m\to \infty$, this encoder achieves at least ``rate'' 0.3233.

\vspace*{.1in}{\sc Acknowledgments:} {This project began while the first-named author was visiting the Indian Institute of Science in Bangalore in August 2024, and he is grateful to the Institute for making this visit possible and for the welcoming atmosphere.  The authors are grateful to Alex Mazel for patiently answering their questions about \cite{mazel2025high} and to Ron Peled for helpful discussions. A.B. and G.W. were partially supported by NSF grants CCF2330909 and CCF2526035. }

\vspace*{.1in} {\small{\sc Competing interests:} The author(s) declare none.}

\vspace*{.1in} {\small {\sc Data availability:} No data are associated with this research.}

\bibliographystyle{amsplain}

\providecommand{\bysame}{\leavevmode\hbox to3em{\hrulefill}\thinspace}
\providecommand{\MR}{\relax\ifhmode\unskip\space\fi MR }
\providecommand{\MRhref}[2]{%
  \href{http://www.ams.org/mathscinet-getitem?mr=#1}{#2}
}
\providecommand{\href}[2]{#2}

\end{document}